\documentclass[a4paper, 11pt]{amsart}   
\usepackage{mathptmx, amssymb,amscd,latexsym, eulervm}   
\usepackage{amsmath}
\usepackage{amsthm}
\usepackage{mathdots}
\usepackage{setspace}
\usepackage{hyperref}
\hypersetup{
 colorlinks,
 linkcolor={blue!90!black},
 citecolor={red!80!black},
 urlcolor={blue!50!black},
breaklinks=true
}
\usepackage{tabularx}
\usepackage{comment}
\usepackage{amsfonts}
\usepackage{paralist}
\usepackage{aliascnt}
\usepackage[initials, lite]{amsrefs}
\usepackage{amscd}
\usepackage{blkarray}
\usepackage{booktabs}
\usepackage{enumitem}
\usepackage{setspace}
\usepackage[inner=2.5cm,outer=2.5cm, bottom=3.2cm]{geometry}
\usepackage{hyperref}
\usepackage[capitalize,nameinlink]{cleveref}
\usepackage{tikz, tikz-cd}
\usepackage{calligra,mathrsfs}
\usetikzlibrary{matrix}
\usetikzlibrary{arrows,calc}
\usepackage{url}

\allowdisplaybreaks

\newtheorem{thm}{Theorem}[section]
\newtheorem{lemma}[thm]{Lemma}
\newtheorem{cor}[thm]{Corollary}
\newtheorem{prop}[thm]{Proposition}

\theoremstyle{definition}
\newtheorem{example}[thm]{Example}
\newtheorem{notation}[thm]{Notation}
\newtheorem{remark}[thm]{Remark}
\newtheorem{definition}[thm]{Definition}

\newtheorem{question}[thm]{Question}
\newtheorem{problem}[thm]{Problem}
\numberwithin{equation}{section}

\newcommand{\Spec}{\mathrm{Spec}}
\newcommand{\Proj}{\mathrm{Proj}}
\newcommand{\Hilb}{\mathrm{Hilb}}
\newcommand{\Gr}{\mathbf{Gr}}

\newcommand{\codim}{\mathrm{codim}}
\newcommand{\LM}{\mathrm{LM}}
\newcommand{\LT}{\mathrm{LT}}
\newcommand{\mdeg}{\mathrm{mdeg}}

\newcommand{\GL}{\mathrm{GL}}
\newcommand{\mult}{\mathrm{mult}}

\newcommand{\eps}{\varepsilon}

\renewcommand{\AA}{\mathbb{A}}
\renewcommand{\P}{\mathbb{P}}
\newcommand{\N}{\mathbb{N}}

\newcommand{\V}{\mathrm{V}}
\newcommand{\kk}{{\Bbbk}}

\newcommand{\bfL}{\mathbf{L}}
\newcommand{\bfM}{\mathbf{M}}
\newcommand{\bfJ}{\mathbf{J}}
\newcommand{\bfS}{\mathbf{S}}

\newcommand{\bfC}{\mathbf{C}}

\newcommand{\bfe}{\mathbf{e}}
\newcommand{\bfv}{\mathbf{v}}
\newcommand{\bfu}{\mathbf{u}}

\newcommand{\mfp}{{\mathfrak{p}}}
\newcommand{\mfq}{{\mathfrak{q}}}

\newcommand{\al}{\alpha}
\newcommand{\be}{\beta}
\newcommand{\ga}{\gamma}
\newcommand{\de}{\delta}

\newcommand{\mcH}{{\mathcal{H}}}
\newcommand{\mcR}{{\mathcal{R}}}
\newcommand{\mcF}{{\mathcal{F}}}
\newcommand{\mcI}{{\mathcal{I}}}
\newcommand{\mcV}{{\mathcal{V}}}
\newcommand{\mcW}{{\mathcal{W}}}

\newcommand{\PLU}{{\mathrm{PLU}}}
\newcommand{\LAP}{{\mathrm{LAP}}}
\newcommand{\UEN}{{\mathrm{UEN}}}
\newcommand{\LEN}{{\mathrm{LEN}}}

\makeatletter
\@namedef{subjclassname@2020}{
  \textup{2020} Mathematics Subject Classification}
\makeatother

\begin{document}

\author[R.\,Ramkumar, A. \,Sammartano]{Ritvik~Ramkumar and Alessio~Sammartano}
\address{(Ritvik Ramkumar) Department of Mathematics\\Cornell University\\Ithaca, NY\\USA}
\email{ritvikr@cornell.edu}
\address{(Alessio Sammartano) Dipartimento di Matematica \\ Politecnico di Milano \\ Milan \\ Italy}
\email{alessio.sammartano@polimi.it}

\title{On Rees algebras of 2-determinantal ideals}



\begin{abstract} 
Let $I$ be the ideal of minors of a $2 \times n$ matrix of linear forms with the expected codimension.
In this paper we prove that the Rees algebra of $I$ and its special fiber ring are Cohen-Macaulay and Koszul; 
in particular, they are quadratic algebras.
The main novelty in our approach is the analysis of a stratification of the Hilbert scheme of determinantal ideals. 
We study  degenerations of Rees algebras along this stratification, and combine it with certain squarefree Gr\"obner degenerations.
\end{abstract}

\maketitle

\section{Introduction}

Let $\bfM$ be a sufficiently general matrix of linear forms, 
and let $I$ be its ideal of maximal minors. 
The study of blowup algebras associated to $I$, in particular the Rees algebra $\mcR(I)= \bigoplus_{k=0}^\infty I^k$, is a central subject in commutative algebra. 
For instance, it appears in the study of rational maps, special varieties, integral dependence, multiplicities, syzygies, and singularities of plane curves;
see \cite{Polini} for an overview.
In all these instances,
one is naturally led to study the defining relations of $\mcR(I)$. 
In doing so, some of the most common problems involve computing the degrees of the defining relations, measuring the singularities of $\mcR(I)$,
and investigating the Koszulness of  $\mcR(I)$. 
These avenues of research are very active and we mention some of the more recent papers
\cite{ALL,BignaletCazalet,BrunsConcaVarbaro,CDFGLPS,CooperPrice,DMN,
Lan,LinShen,LinShen2,Sammartano}.

In this work, we provide a detailed picture for the case of matrices  with two rows.
Many interesting geometric objects are defined by minors of such matrices,  
for instance rational normal scrolls \cite{EisenbudHarris}, 
2-regular algebraic sets, small schemes \cite{EGHP},
and eigenschemes of square matrices  \cite{AEKP}.
Moreover, the rational maps associated with some of these objects exhibit rich geometry as seen in the work \cite{RussoSimis}.
For simplicity, 
we say that an ideal $I$ is \emph{2-determinantal} if it is generated by the $2\times 2$ minors of a $2\times(c+1)$ matrix of linear forms, where  $c=\codim(I)$.
Our main result is the following:

\begin{thm}\label{TheoremMain}
Let $I$ be a 2-determinantal ideal.
The Rees algebra $\mcR(I)$ and the special fiber ring $\mcF(I)$ are Cohen-Macaulay Koszul algebras.
In particular, they are defined by quadratic relations.
\end{thm}

The theorem builds on and  generalizes several previous results. 
When $I$ defines a balanced rational normal scroll,
it was proved in 
\cite{CHV} that $\mcR(I)$ and  $\mcF(I)$ are Koszul and Cohen-Macaulay (in fact, that they have rational singularities).
The results of  \cite{BrunsConcaVarbaro} show  that the powers of any 2-determinantal ideal $I$ have a linear free resolution, a property that is closely related to the Koszulness of the Rees algebra. 
Moreover, they show that $I$ is an ideal of fiber type, 
i.e., the defining equations of $\mcR(I)$ consist of those of $\mcF(I)$ and the first syzgies of $I$.
When $I$ defines an arbitrary rational normal scroll, a part of the theorem was established in the works of \cite{Sammartano} and \cite{LinShen}. 
In \cite{Sammartano}, the author proves the Koszulness of $\mcR(I)$ and $\mcF(I)$ by explicitly constructing Gr\"obner bases for their defining ideals. 
In \cite{LinShen}, the authors prove that $\mcF(I)$ is Cohen-Macaulay by showing that the simplicial complex determined by 
 the Gr\"obner basis of 
 \cite{Sammartano} is shellable.

The proof of Theorem \ref{TheoremMain} relies on a variety of techniques and, unlike the previously established results, has relatively little case analysis. 
The central idea is to seek convenient degenerations of  Rees algebras of arbitrary 2-determinantal ideals that are algebraically simpler.
This naturally leads one to examine stratifications of the Hilbert scheme of 2-determinantal ideals.
The reason we believed that such an approach would be fruitful is due to a theorem of  Harris \cite{Harris},
which completely characterizes  degenerations among   scrolls by means of their integer partitions.
Notably, Harris shows that, for each dimension and codimension, 
there exists a unique scroll,
corresponding to the least balanced partition,
such that every other scroll degenerates to it.
Using the theory of Kronecker-Weierstrass normal forms, 
we investigate the analogous problem for the broader class of 
2-determinantal ideals.
For each dimension and codimension, we identify a distinguished 2-determinantal ideal $L$
with the property that every  2-determinantal  (non-cone)
ideal  degenerates to it, see Theorem \ref{TheoremUniqueMinimumHcd} for a more precise statement. 
The scheme defined by $L$ is a union of two components, 
a smooth Segre variety and a non-reduced 2-regular scheme supported on a linear space, 
that are linearly joined along a codimension one linear space, in the sense of \cite{EGHP}.

This  fact, combined with results of \cite{BrunsConcaVarbaro},
reduces the proof of Theorem \ref{TheoremMain}  to the study of the blowup algebras of the single ideal $L$,
for each dimension and codimension. To conclude, we employ the combinatorial techniques developed in \cite{Sammartano} for the case of  scrolls to 
prove that both $\mcF(L)$ and $\mcR(L)$ have a squarefree Gr\"obner basis of quadrics  (Theorems \ref{TheoremSpecialFiber} and \ref{TheoremReesAlgebraGB}).
Although the proof outline is similar, 
the combinatorial analysis is substantially trimmer,
chiefly because we only deal with one explicit ideal as opposed to all possible scroll partitions in \cite{Sammartano},
and,
thus, we obtain a  shorter proof even for  scrolls.

A  more significant consequence of our analysis of degenerations concerns the Cohen-Macaulayness of the blowup algebras.
This property was known to hold for both blowup algebras in the case of balanced scrolls \cite{CHV} via Sagbi bases,
and for the special fiber ring of arbitrary scrolls \cite{LinShen} via a detailed analysis of the initial complex. 
Combining our results  with the  theorem of Conca-Varbaro \cite{ConcaVarbaro},
we deduce effortlessly the Cohen-Macaulayness for both blowup algebras and for all 2-determinantal ideals.

\section{Preliminaries} \label{SectionPreliminaries}

Throughout this work, $\kk$ denotes a field.
In some sections, we assume that $\kk$ is algebraically closed, but the main theorem will be proved for arbitrary fields.
We use the symbol  $[\cdot]_d$ to denote a graded component of degree $d$.

\subsection{Determinantal ideals}

Let $S=\mathrm{Sym}(\kk^{n+1})$ be a  polynomial ring and  $\P^n = \Proj \, S$ its projective space.
Given a matrix $\bfM$, we denote by $I_t(\bfM)$ the ideal generated by all the $t\times t$ minors of $\bfM$.
We fix a  terminology for the ideals that are the main subject of this paper:

\begin{definition}
An ideal $I \subseteq S$ is called  {\bf 2-determinantal ideal} if there exists a $2 \times (c+1)$ matrix $\bfM$ of linear forms such that $I = I_2(\bfM)$ and  
$\codim(I)=c$.
\end{definition}

Analogously,
the closed subscheme  $\V(I) \subseteq \P^n$  defined by a 2-determinantal ideal $I\subseteq S$ is called a  2-determinantal scheme.
These objects enjoy many well-known properties. 
The condition that $\codim(I)=c$ means that $I$ has the generic codimension with respect to the Eagon-Northcott bound, and this is equivalent to the fact that the Eagon-Northcott complex associated to $\bfM$ is a (minimal, linear) free resolution of $S/I$.
It follows that  $S/I$ is a Cohen-Macaulay ring of (minimal) multiplicity $c+1$,  and,
if $I$ is prime and $\kk=\overline{\kk}$, 
then $\V(I)$ is a variety of minimal degree.
See \cite{BrunsVetter} for details.

Now, assume that $\kk  =\overline{\kk}$.
For fixed integers $c$ and $n$, all 2-determinantal schemes have numerically the same free resolution. Hence, they have the same Hilbert function and the same Hilbert polynomial $p(\zeta)$, and, as a consequence, they are parametrized by the same Hilbert scheme $\Hilb^{p(\zeta)}(\P^n)$.
We point out that there is a vast literature on Hilbert schemes of determinantal schemes, in a much more general setting, see e.g. \cite{KleppeMiroRoig},
but we will only be concerned with the 2-determinantal case.

\subsection{Blowup algebras}
Let $I= I_2(\bfM)=(\delta_1, \ldots, \delta_m) \subseteq S$
be a 2-determinantal ideal, where the $\delta_i$ are the $2 \times 2$ minors of $\bfM$.
The {\bf Rees algebra} of $I$ is the ring $\mcR(I)=S[\delta_1 \tau, \ldots, \delta_m \tau] \subseteq S[\tau]$ where $\tau$ is a variable.
It is a  standard bigraded ring by setting $\deg(x_i) = (1,0), \deg(\tau)=(-2,1)$, so that $\deg(\delta_i \tau)=(0,1)$.
The {\bf special fiber ring} of $I$ is the subring $\mcF(I)\subseteq \mcR(I)$ concentrated in bidegrees $(0,\ast)$,
and it can also be identified with the subring
$\kk[\delta_1 , \ldots, \delta_m ] \subseteq S$.
The name ``special fiber'' comes from the fact that $\mcF(I) \cong \mcR(I) \otimes_S \kk$.
Introducing variables $T_1, \ldots, T_m$, we have a commutative diagram 
\begin{center}
\begin{tikzcd}
P_\mcR = S[T_1,\ldots, T_m]  \arrow{r}{\pi_\mcR} \arrow[swap]{d}{\rho'} & \,\,\mcR(I) \arrow[swap]{d}{\rho}  \\
P_\mcF = \kk[T_1,\ldots, T_m]  \arrow{r}{\pi_\mcF} \arrow[xshift=0.3cm,swap]{u}{\iota'} & \,\, \mcF(I) \arrow[xshift=0.3cm,swap]{u}{\iota}
\end{tikzcd}
\end{center}
where $\rho, \rho'$ are the algebra retractions defined by projecting onto bidegrees $(0,\ast)$,
$\iota, \iota'$ are the natural inclusions, and $\pi_\mcR, \pi_\mcF$ are  the (bi)graded surjective maps defined by $\pi_\mcR(T_i)= \delta_i \tau$.
The ideals $\ker(\pi_\mcR)$ and $\ker(\pi_\mcF)$ are called the {\bf defining ideals} of $\mcR(I)$ and $\mcF(I)$, respectively.
The ideal $I$ is said to be of {\bf fiber type} if $\ker(\pi_\mcR)$ is generated in bidegrees $(\ast, 1)$ and $(0, \ast)$, equivalently, by the polynomials arising from the first syzygies of $I$ and by $\ker(\pi_\mcF)$.

We refer to \cite{Polini} for an overview on defining ideals of blowup algebras.

\subsection{Degenerations and singularities}
In this subsection, we assume  $\kk=\overline{\kk}$.
We discuss the kind of flat families we will be primarily interested in.
Let $B$ be the localization $\kk[t]_f$ at some  polynomial $f\in \kk[t]$ such that $f(0) \ne 0$,
that is, 
$\Spec(B) \subseteq \AA^1$ is an  open subset containing $0$.
We define a {\bf one-parameter family  of graded  algebras} to be an  $\N$-graded $B$-algebra $A$ 
such that $[A]_0 = B$ and $A$ is generated by $[A]_1$.
For $\alpha \in \Bbbk$, we denote by $A_\alpha = A/(t-\alpha)$ the corresponding member  of the family. 

\begin{lemma}\label{LemmaFlatnessHilbertFunction}
A one-parameter family $A$ of graded  algebras over $B$  is  flat if and only if all the members $A_\alpha$ have the same Hilbert function. 
\end{lemma}
\begin{proof}
This can be found for example in \cite[Exercise 20.14]{Eisenbud}.
\end{proof}

Given  standard graded $\kk$-algebras $R, R'$ we say that $R$ {\bf degenerates} to $R'$ 
if there exists a flat one-parameter family $A$ such that $A_0 \cong R'$ and $A_\alpha \cong R$ for  $\alpha \ne 0$. 
By abuse of terminology, given  homogeneous ideals $I,I'\subseteq S$, we say that $I$  degenerates to $I'$ if 
$S/I$ degenerates to $S/I'$. 
A notable special case is the degeneration arising from a Gr\"obner basis \cite[Section 15.8]{Eisenbud}.
However,
it is often fruitful to seek degenerations in the  generality described above;
in fact, this is one of the key aspects of this paper.

Several singularities of  algebras behave well in flat families;
we list the ones that are relevant to our main results.
A standard graded $\kk$-algebra $R$ is called {\bf Koszul} if $[\mathrm{Tor}^R_i(\kk,\kk)]_j = 0$ whenever $i \ne j$,
that is, if the field $\kk$ has a linear free $R$-resolution.
Koszul algebras are defined by quadrics,  whereas algebras defined by Gr\"obner bases of quadrics are Koszul.
See \cite{CDR} for more details on Koszul algebras.

A Noetherian local ring $(R, \mathfrak{m})$ is called {\bf cohomologically full} if,
 for every local ring $(T,\mathfrak{n})$  such that $\mathrm{char}(R)= \mathrm{char}(T)$ and 
 $\mathrm{char}(R/\mathfrak{m})= \mathrm{char}(T/\mathfrak{n})$,
 and every surjection $T \to R$ such that the induced map $T/\sqrt{0}\to R/\sqrt{0}$ is an isomorphism, 
 the induced maps on local cohmology $H^\bullet_\mathfrak{n}(T) \to H^\bullet_\mathfrak{m}(R)$ are surjective.
A standard graded $\kk$-algebra $R$ is called cohomologically full if its localization at the homogeneous maximal ideal is cohomologically full.
See \cite{DDM} for details on cohomologically full rings.

\begin{lemma}\label{LemmaSingularitiesDegeneration}
Let $R,R'$ be standard graded $\kk$-algebras such that 
$R$ degenerates to $R'$.
\begin{enumerate}
\item If $R'$ is Cohen-Macaulay, then $R$ is Cohen-Macaulay;
\item If $R'$ is Koszul, then $R$ is Koszul;
\item If $R'$ is cohomologically full, then $R$ is cohomologically full.
\end{enumerate}
\end{lemma}
\begin{proof}
Item (1) is well-known. 
Item (2) is well-known in the case of Gr\"obner degenerations, see e.g. \cite[Lemma 2.1]{Caviglia},
but the argument works for any one-parameter degeneration in the above sense.
Similarly, item (3) is proved in \cite[Proposition 3.3]{ConcaVarbaro} in the case of square-free Gr\"obner degenerations, 
but the argument works for any one-parameter degeneration  and any cohomologically full ring $R'$.
\end{proof}

By \cite{ConcaVarbaro}, the converse of item (1) in Lemma \ref{LemmaSingularitiesDegeneration} holds when $R'$ is a Stanley-Reisner ring. 
More generally, we have:
 
\begin{lemma}\label{LemmaCohomologicallyFullDegenerations}
Let $R,R'$ be standard graded $\kk$-algebras such that 
$R$ degenerates to $R'$.
If $R'$ is cohomologically full and $R$ is Cohen-Macaulay, 
then $R'$ is Cohen-Macaulay.
\end{lemma}
\begin{proof}
The statement follows from  \cite[Theorem 1.2]{ConcaVarbaro}.
While the theorem is stated for Gr\"obner degenerations to a Stanley-Reisner ring $R'$,
the argument works for any cohomologically full ring $R'$  \cite[Remark 2.6]{ConcaVarbaro}
and any one-parameter degeneration.
\end{proof}

\section{The Kronecker-Weierstrass type }\label{SectionKWtype}

The goal of this section is 
to study the structure of 2-determinantal ideals.
In particular, we  introduce a discrete invariant that we call the Kronecker-Weierstrass type, or simply KW type.
It is based on the Kronecker-Weierstrass theory of matrices, which we review in this section.
In order to define this invariant, we will describe the primary decomposition of  2-determinantal ideals.
Throughout this section, we assume $\kk= \overline{\kk}$.

Let $S$ be a polynomial ring over $\kk$.
A {\bf scroll block} is a matrix of the form 
$$
\begin{pmatrix}
x_1 & x_2 & x_3 & \cdots & x_{p-2} & x_{p-1} & x_{p} \\
x_2 & x_3 & x_4 & \cdots & x_{p-1} & x_{p} & x_{p+1} 
\end{pmatrix}
$$
for some variables $x_1, \ldots, x_{p+1}$.
A {\bf Jordan block} is a matrix of the form 
$$
\begin{pmatrix}
x_1 & x_2 & x_3 & \cdots & x_{p-2} & x_{p-1} & x_{p} \\
x_2 + \eps x_1& x_3+ \eps x_2 & x_4+ \eps x_3 & \cdots & x_{p-1}+ \eps x_{p-2} & x_{p}+ \eps x_{p-1} & \eps x_p
\end{pmatrix}
$$
where $\eps \in \kk$ is  called the {\bf eigenvalue} of the Jordan block.
A {\bf nilpotent  block} is a matrix of the form 
$$
\begin{pmatrix}
0 & x_1 & x_2 & x_3 & \cdots & x_{p-2} & x_{p-1} & x_{p} \\
 x_1 & x_2 & x_3 & x_4 & \cdots & x_{p-1} & x_{p} & 0
\end{pmatrix}.
$$
A {\bf concatenation} of blocks is a matrix 
$\bfM = (\,\bfM_1\, |\, \bfM_2 \,|\, \cdots \,|\,\bfM_e)$ where each 
$\bfM_i$ is a scroll, Jordan, or nilpotent block
and the sets of variables appearing in different blocks are disjoint.

\begin{prop}\label{PropExistenceConcatenation}
Let $\bfM$  be  a $2\times (c+1)$ matrix of linear forms in $S$.
Using row and column operations and  linear automorphisms of $S$, one can transform  $\bfM$ into a matrix $\bfM'$  that is a concatenation of blocks.
\end{prop}
\begin{proof}
See for instance \cite[Section 3]{CatalanoJohnson}.
\end{proof}

\begin{definition}\label{DefKWNormalForm}
Let $I \subseteq S$ be a 2-determinantal ideal with $\codim(I)=c$.
A {\bf Kronecker-Weierstrass normal form}  of $I$ is a
$2\times(c+1)$ matrix $\bfM$ 
that is a concatenation of blocks and satisfies 
$\varphi(I)=I_2(\bfM)$ for some linear automorphism $\varphi$ of $S$.
\end{definition}

By Proposition \ref{PropExistenceConcatenation}, such a matrix $\bfM$ always exists.
We emphasize that our definition includes the requirement that $\bfM$ has the least possible number of columns $c+1$.
For instance, the matrices
$$
\begin{pmatrix}
x_1 & x_2 & x_3 & 0 \\
x_2 & x_3 & x_4 & 0 
\end{pmatrix}
\quad 
\text{and}
\quad
\begin{pmatrix}
0 & x_1 & 0 & x_2 \\
x_1 & 0 & x_2 & 0 
\end{pmatrix}
$$
are not  Kronecker-Weierstrass normal forms,
since their ideals of minors have codimension $c=2$.
Below, we classify all 
the Kronecker-Weierstrass normal forms of 2-determinantal ideals.

\begin{prop}\label{PropKWNormalFormMaximalCodimension}
Let $\bfM$ be a $2\times(c+1)$ matrix that is a concatenation of blocks and let $I=I_2(\bfM)$.
Then, $\codim(I) = c$ if and only if one of the following two conditions holds:
\begin{enumerate}
\item
$\bfM$ consists of scroll blocks and Jordan blocks with distinct eigenvalues.
\item
$\bfM$ consists of one nilpotent block.
\end{enumerate}
\end{prop}
\begin{proof}
This follows  from \cite{Chun}, where the Hilbert series is calculated for all  concatenations of blocks. 
Specifically, the desired statement follows from 
\cite[Theorems 2.2.3 and 2.5.5]{Chun}.
\end{proof}

Case (2) should be considered as an exceptional case in this classification. 
Ideals of this form are simply 
  the squares of linear primes.
Equivalently, they are  the 2-determinantal ideals $I_2(\bfM)$ such that $\codim \, I_2(\bfM) = \codim\, I_1(\bfM)$.
In this paper, we are  concerned with the ideals of the form (1).

Kronecker-Weierstrass normal forms are not invariant under projective equivalence:
 the ``continuous'' data, i.e., the eigenvalues of the Jordan blocks, is not uniquely determined.
A simple example is given by the concatenation of Jordan blocks
$$
\begin{pmatrix}
x_1 & x_2 & x_3 & x_4 \\
\eps_1 x_1 & \eps_2  x_2 &\eps_3 x_3 & \eps_4 x_4 
\end{pmatrix},
$$ 
whose ideal of minors is $(x_1x_2, x_1x_3, x_1x_4, x_2x_3, x_2x_4, x_3x_4)$ as long as the eigenvalues are all distinct.
Another example is the following result,
which will be useful later.

\begin{lemma}\label{LemmaTranslatingEigenvalues}
Let $\bfM$ be a concatenation of scroll blocks and Jordan blocks. 
Given $\delta \in \kk$, let $\bfM'$  denote the matrix obtained from $\bfM$ by adding $\delta$ to all the eigenvalues of the Jordan blocks.
Then, $I_2(\bfM)$ and $I_2(\bfM')$ are the same ideal  up to a 
change of coordinates involving only the variables of the scroll blocks.
\end{lemma}
\begin{proof}
This is proved in \cite[pp. 42--43]{CatalanoJohnson}.
\end{proof}

On the other hand, 
we are going to see that
the discrete data is  invariant under projective equivalence.

When   $\bfM$ is a  concatenation of scroll blocks, the ideal $I=I_2(\bfM)$ is  prime and  $\V(I)$ is a  variety of minimal degree, called a {\bf rational normal scroll}.

\begin{lemma}\label{LemmaScrollUniqueness}
Let $\bfM$ and $\bfM'$ be concatenations of scroll blocks 
of sizes $a_1 \geq \cdots \geq a_d$ 
and $a'_1 \geq \cdots \geq a'_{d'}$.
Then,  $I_2(\bfM)$ and $I_2(\bfM')$ are projectively equivalent if and only if
$d=d'$ and $a_i = a'_i$ for every $i$.
\end{lemma}
\begin{proof}
This is well-known, see for instance  \cite{Harris}. 
\end{proof}

We now describe the primary decomposition of the ideals of the form  (1) in Proposition \ref{PropKWNormalFormMaximalCodimension}.
With a slight abuse of terminology, by {\bf multiplicity} of an ideal $I \subseteq S$ we mean the degree of the subscheme $\V(I)$,
equivalently, the multiplicity of the ring $S/I$, and we denote this number by $\mult(I)$.

\begin{thm}\label{ThmPrimaryDecomposition}
Let $ \bfM = (\,\bfS_1 \,|\, \cdots \,|\, \bfS_d \,|\, \bfJ_1 \,|\, \cdots \,|\, \bfJ_e )  $ be a  concatenation of scroll blocks
 $\bfS_1,\dots,\bfS_d$ and Jordan blocks  $\bfJ_1,\dots,\bfJ_e$  with distinct eigenvalues. 
The ideal  $I=I_2(\bfM)$ has  $e+1$ primary components. 
One  component is the 
prime ideal
$
\mfp_0 = 
I_2
(\,\bfS_1 \,|\, \cdots \,|\, \bfS_d )
+
I_1
(\, \bfJ_1 \,|\, \cdots \,|\, \bfJ_e ),
$ 
which defines a  rational normal scroll.
For each $1 \leq i \leq e$, there is a component 
$\mfq_i$
primary to  a linear prime and with multiplicity equal to the number of columns of $\,\bfJ_i$.
\end{thm}

\begin{proof}
Assume $\bfM$ has $c+1$ columns. 
Since $\codim(I)=c$ by Proposition \ref{PropKWNormalFormMaximalCodimension},
  the ring $S/I$ is  Cohen-Macaulay of multiplicity $c+1$.
It follows that $I$ is unmixed, i.e. all associated primes are minimal and have codimension $c$,
and the multiplicities  of the primary components  add up to $c+1$.

We fix a notation for the variables in the blocks of $\bfM$ by setting
$$
\bfS_i = 
\begin{pmatrix}
x_{i,1} & x_{i,2}  & \cdots &  x_{i,a_i} \\
x_{i,2} & x_{i,3} & \cdots  & x_{i,a_i+1} 
\end{pmatrix} \quad \text{and} \quad
\bfJ_i = 
\begin{pmatrix}
y_{i,1} &  \cdots & y_{i,b_i-1} & y_{i,b_i} \\
y_{i,2} + \eps_i y_{i,1}  & \cdots & y_{i,b_i}+ \eps_i y_{i,b_i-1} & \eps_i y_{i,b_i}
\end{pmatrix}.
$$
In particular, we have $c +1= \sum_{i=1}^d a_i + \sum_{i=1}^e b_i$.
Up to adjoining variables, we may assume that the ambient polynomial ring is $S = \kk[x_{1,1},\dots,x_{d,a_d+1},y_{1,1},\dots,y_{e,b_e}]$.

Consider the ideals 
$J' = I_2
(\,\bfS_1 \,|\, \cdots \,|\, \bfS_d)
$
and 
$
J'' = 
I_1
(\, \bfJ_1 \,|\, \cdots \,|\, \bfJ_e \,).
$ 
They are the defining ideals of   a rational normal scroll and a linear space, respectively. 
It follows that they are both prime and that
$$\codim(J') +1= \mult(J')=\sum_{i=1}^d a_i,\quad \codim(J'') = \sum_{i=1}^e b_i, \quad \text{and} \quad \mult(J'')=1.
$$
Since $J'$ and $J''$ involve disjoint sets of variables, 
the ideal $\mfp_0 = J_1+J_2$ is prime with $\codim(\mfp_0) = \codim(J_1)+ \codim(J_2) = c$ and $\mult(\mfp_0) = \mult(J_1) \mult(J_2)= \sum_{i=1}^d a_i$.
We conclude that $\mfp_0$ is a minimal prime of $I$, since $I\subseteq \mfp_0$ and $\codim(I)=\codim(\mfp_0)$.
Let $\mfq_0$
denote the $\mfp_0$-primary component of $I$. Then,
$\mult(\mfq_0) \geq \mult(\mfp_0)$, and
$\mult(\mfq_0) = \mult(\mfp_0)$ if and only if
$\mfq_0=\mfp_0$.

We claim that,
 for every Jordan block $\bfJ_i$,
  there is a  linear minimal prime $\mfp_i$ of $I$ such that $I_1(\bfJ_i) \not\subseteq \mfp_i$ and 
$I_1(\bfJ_h) \subseteq \mfp_i$ for all $h \ne i$,
and such that the $\mfp_i$-primary component $\mfq_i$ has $\mult(\mfq_i) = b_i$.
This claim implies the conclusion of the theorem.
In fact, it follows from the claim that the minimal primes $\mfp_0, \ldots, \mfp_e$ are all  distinct.
Moreover, since the sum of the multiplicities of all the primary components of $I$ must equal $\mult(I) = 
 \sum_{i=1}^d a_i + \sum_{i=1}^e b_i$,
  we can
  conclude   
 that  $\mult(\mfq_0) = \sum_{i=1}^d a_i$, and hence $\mfq_0=\mfp_0$. 
 This shows that $\mfq_0, \mfq_1,\dots,\mfq_e$
are the only primary components.

It suffices to prove the claim for one Jordan block, say $\bfJ_1$.
Applying Lemma \ref{LemmaTranslatingEigenvalues},
 we may assume that its eigenvalue is $\eps_1 = 0$.
Let $\mfp_1$ be the linear prime generated by the second row of $\bfM$, equivalently,
$$
\mfp_1 = 
\sum_{i=1}^d\big(x_{i,2},\dots,x_{i,a_i+1}\big) + 
\big(y_{1,2}, \ldots, y_{1,b_1}\big) +
 \sum_{i=2}^e\big(y_{i,1},\dots,y_{i,b_i}\big).
 $$
Then $\codim(\mfp_1) = \sum_{i=1}^d a_i + (b_1-1) + \sum_{i=1}^eb_i= c = \codim(I)$ and $I = I_2(\bfM)\subseteq \mfp_1$,
so $\mfp_1$ is a minimal prime.
It remains to show that $\mult(\mfq_1) = b_1$,
equivalently,
that the Artinian ring $(S/I)_{\mfp_1}$ has length  $b_1$.
Observe that $\mathbb{K}=\kk(x_{1,1},x_{2,1}\dots,x_{d,1},y_{1,1})\subseteq (S/I)_{\mfp_1}$. 
We are going to show that 
\begin{equation}\label{EqIsomorphismLength}
(S/I)_{\mfp_1} \cong {\mathbb{K}[z]}/{\big(z^{b_1}\big)}
\end{equation}
thereby concluding  the proof.

Let $a = \sum_{i=1}^d a_i$ be the number of scroll columns.
The variable  $y_{1,1}$, in position $(1,a+1)$ in $\bfM$, 
is invertible in $S_{\mfp_1}$.
Using row and  column operations we reduce  $\bfM$ to
$$
\begin{pmatrix}
0 & \cdots & 0 & 1				 & 0 	 		& \cdots &  0\\
\Delta_1 & \cdots & \Delta_a &  0 & \Delta_{a+2} 	& \cdots  & \Delta_{c+1}
\end{pmatrix},
$$
where $\Delta_i$ is the  $2\times 2$ minor of $\bfM$ 
involving columns $i$ and  $a+1$.
Thus,
 we have
$$
I_{\mfp_1} = 
\big(\Delta_1, \ldots, \Delta_a, \Delta_{a+2},\ldots, \Delta_{c+1}\big)
S_{\mfp_1}.
$$
It is convenient  to relabel the minors with respect to variables and blocks they appear in. Let
$$
z = \frac{y_{1,2}}{y_{1,1}}, \quad
\Delta^x_{i,j} := x_{i,j+1}  - zx_{i,j}, \quad \text{and} \quad
\Delta^y_{i,j} = 
\begin{cases}
y_{i,j+1} +(\eps_i-z)y_{i,j}  & \text{ if }  j < b_i, \\
(\eps_i-z)y_{i,b_i}  & \text{ if } j = b_i.
\end{cases}
$$
We have 
\begin{equation}\label{EqISum}
I_{\mfp_1}
= 
\sum_{i=1}^d \big(\Delta^x_{i,1}, \ldots, \Delta^x_{i,a_i}\big)
+ 
\big(\Delta^y_{1,2},\ldots, \Delta^y_{1,b_1}\big)
+
\sum_{i=2}^e 
\big(\Delta^y_{i,1},\ldots, \Delta^y_{i,b_i}\big).
\end{equation}
For each scroll block $\bfS_i$ we may rewrite
\begin{equation*}
\begin{aligned}
\big(\Delta^x_{i,1}, \ldots, \Delta^x_{i,a_i}\big) 
&=
\big(x_{i,2}-zx_{i,1},\,x_{i,3}-zx_{i,2},\, \ldots,\, x_{i,a_i+1}-zx_{i,a_i}\big) \\
&=
\big(x_{i,2}-zx_{i,1},\,x_{i,3}-z^2x_{i,1},\, \ldots, \,
x_{i,a_i+1}-z^{a_i}x_{i,1}\big).
\end{aligned}
\end{equation*}
For the Jordan block $\bfJ_1$ we have
\begin{equation*}
\begin{aligned}
\big(\Delta^y_{1,2}, \ldots, \Delta^y_{1,b_1}\big) 
&=
\big(y_{1,3}-zy_{1,2},\, y_{1,4}-zy_{1,3},\, \ldots, \,
y_{1,b_1}-zy_{1,b_1-1},\,
-zy_{1,b_1}\big) 
\\
&=
\big(y_{1,3}-zy_{1,2},\, y_{1,4}-z^2y_{1,2},\, \ldots, \,
y_{1,b_1}-z^{b_1-2}y_{1,2},\,
zy_{1,b_1}
\big) 
\\
&=
\big(y_{1,3}-zy_{1,2},\, y_{1,4}-z^2y_{1,2},\, \ldots, \,
y_{1,b_1}-z^{b_1-2}y_{1,2},\,
z^{b_1-1}y_{1,2}
\big)
\\
&=
\big(y_{1,3}-zy_{1,2},\, y_{1,4}-z^2y_{1,2},\, \ldots, \,
y_{1,b_1}-z^{b_1-2}y_{1,2},\,
z^{b_1}y_{1,1}
\big).
\end{aligned}
\end{equation*}
For the Jordan blocks $\bfJ_i$ with $i>1$ we consider the unit $\gamma_i = (z-\eps_i ) \in S_{\mfp_1}$ and we have
\begin{equation*}
\begin{aligned}
\big(\Delta^y_{i,1}, \ldots, \Delta^y_{1,b_i}\big) 
&=
\big(y_{i,2}-\gamma_iy_{i,1},\, y_{i,3}-\gamma_iy_{i,2},\, \ldots, \,
y_{i,b_i}-\gamma_iy_{i,b_i-1},\,
-\gamma_iy_{i,b_i}\big) 
\\
&=
\big(y_{i,2}-\gamma_iy_{i,1},\, y_{i,3}-\gamma_i^2y_{i,1},\, \ldots, \,
y_{i,b_i}-\gamma_i^{b_i-1}y_{i,1},\,
y_{i,b_i}\big) 
\\
&=
\big(y_{i,2},\, y_{i,3},\, \ldots, y_{i,{b_i-1}}\,
y_{i,1},\,
y_{i,b_i}\big).
\end{aligned}
\end{equation*}
Thus, we may apply a local change of coordinates in the ring $S_{\mfp_1}$ carrying
\begin{eqnarray*}
\big(\Delta^x_{1,1}, \ldots, \Delta^x_{1,a_1}\big) 
&\quad\text{to}\quad &
\big(x_{i,2},\,x_{i,3},\, \ldots, \,
x_{i,a_i+1}\big),
\\
\big(\Delta^y_{1,2}, \ldots, \Delta^y_{1,b_1}\big) 
&\quad\text{to}\quad &
\big(y_{1,3},\, y_{1,4},\, \ldots, \,
y_{1,b_1},\,
z^{b_1}
\big),
\\
\big(\Delta^y_{i,1}, \ldots, \Delta^y_{1,b_i}\big) 
&\quad\text{to}\quad &
\big(y_{i,1},\, y_{i,2},\, \ldots, \,
y_{i,b_i}\big) \quad \text{for} \quad i >1.
\end{eqnarray*} 
The desired conclusion \eqref{EqIsomorphismLength} follows from \eqref{EqISum}.
\end{proof}

It follows from Proposition \ref{PropKWNormalFormMaximalCodimension},
Lemma \ref{LemmaScrollUniqueness},
 and Theorem \ref{ThmPrimaryDecomposition} that the Kronecker-Weierstrass normal form of a 2-determinantal ideal $I$ is uniquely determined, up to the eigenvalues  and rearrangement of the blocks.
 Thus,
we may give the following definition of Kronecker-Weierstrass type.
In this paper,  an {\bf integer partition} is a sequence of integers $\lambda = (\lambda_1, \ldots, \lambda_d)$ such that $\lambda_1 \geq \cdots \geq \lambda_d >0$.
We use the exponential notation to denote repetitions, for example, $(3^2,2,1^4) = (3,3,2,1,1,1,1)$.

\begin{definition}\label{DefKWtype}
Let $\lambda=(\lambda_1, \ldots, \lambda_d)$ and $\mu=(\mu_1, \ldots, \mu_e)$
be integer partitions.
An ideal $I\subseteq S$ has {\bf Kronecker-Weierstrass type} $(\lambda;\mu)$ if
there is a linear automorphism $\varphi$ of $S$ such that $\varphi(I) = I_2(\bfM)$,
for some concatenation  
$\bfM = 
(\,\bfS_1 \,|\, \cdots \,|\, \bfS_d \,|\, \bfJ_1 \,|\, \cdots \,|\, \bfJ_e )  $  of scroll blocks $\bfS_i$ of size $\lambda_i$ and Jordan blocks 
$\bfJ_i$ of size $\mu_i$ with distinct eigenvalues.
An ideal $I\subseteq S$ is of {\bf nilpotent Kronecker-Weierstrass type} if it is the square of a linear prime.
\end{definition}

We will often abbreviate the phrase ``Kronecker-Weierstrass type'' to ``KW type''.

We point out that the Kronecker-Weierstrass type is not a \emph{complete} invariant of projective equivalence.
In other words, unlike the case of rational normal scrolls,
there exist  2-determinantal ideals with the same KW  type that are not projectively equivalent.

\begin{example}
We consider  ideals of KW type $(1,1;1^n)$.
For each set 
 $\mathbf{e} = \{\eps_1, \ldots, \eps_n\} $ of distinct eigenvalues,  consider the 
Kronecker-Weierstrass normal form
$$
\bfM(\mathbf{e}) = 
\begin{pmatrix}
x_1 & x_3 & y_1 & y_2 & \cdots & y_n \\
x_2 & x_4 & \eps_1 y_1 & \eps_2 y_2 & \cdots & \eps_n y_n \\
\end{pmatrix}.
$$
The ideal $I(\mathbf{e}) = I_2(\bfM(\mathbf{e}))$ has codimension $n+1$ and multiplicity $n+2$.
By Theorem \ref{ThmPrimaryDecomposition},
 it has $n+1$ primary components,  and they are all prime.
In fact, from the proof we see that they are
\begin{eqnarray*}
\mfp_0 &=& (x_1x_4-x_2x_3, y_1, \ldots, y_n) \quad \text{and} \\
\mfp_i(\mathbf{e}) &=& (x_2- \eps_i x_1, x_4 -\eps_i x_3,  y_1, \ldots, y_{i-1}, y_{i+1},y_n) \quad \text{for} \quad i = 1, \ldots, n.
\end{eqnarray*}
The linear span  of the quadric surface  $\V(\mfp_0)$ 
is  $\P^3=\V(y_1, \ldots, y_n)$. 
Consider the lines 
$$
\ell_i(\mathbf{e}) =  \V(\mfp_i(\mathbf{e})) \cap \V(\mfp_0) = \V(\mfp_i(\mathbf{e})) \cap \P^3 = \V(x_2- \eps_i x_1, x_4 -\eps_i x_3)\subseteq \P^3
$$
and their collection $L(\mathbf{e}) = \big\{\ell_1(\mathbf{e}), \ldots, \ell_n(\mathbf{e})\big\}$.
As $\mathbf{e}$ varies, $L(\mathbf{e})$ describes a locus $\mathcal{L}\subseteq \mathrm{Sym}^n\, \Gr(1,\P^3)$ of dimension $n$.
Now, suppose there exists a linear automorphism $\varphi$ of $S$
such that  $\varphi(I(\mathbf{e}))= I({\mathbf{e}'})$ for  two tuples $\mathbf{e}, \mathbf{e}'$.
Then $\varphi$  must fix the scroll component $\V(\mfp_0)$ and, hence,
 its linear span,
so it restricts to a linear automorphism $\overline{\varphi}$ of $\P^3$.
Moreover, $\varphi$ 
must carry the components $\{\mfp_1(\bfe), \ldots, \mfp_n(\bfe)\}$ to  $\{\mfp_1(\bfe'), \ldots, \mfp_n(\bfe')\}$,
 therefore,  $\overline{\varphi}$ must carry 
 $L(\bfe)$ to $L(\bfe')$.
 If $n \gg 0$, it is clear, by dimension considerations, that $\mathcal{L}$ cannot be a single $\mathrm{GL}_4(\kk)$-orbit. Thus, for general $\mathbf{e},\mathbf{e}'$, there will be no such linear automorphism $\varphi$.
\end{example}

\begin{remark}
The work \cite{AEKP} contains 
 results related to the content of this section.
Roughly speaking, the authors  describe the primary decomposition of the determinantal ideals of a matrix consisting only  of Jordan blocks, but with possibly repeated eigenvalues.
\end{remark}

\section{Degenerations of 2-determinantal ideals}
\label{SectionDegenerations}

Let $S=\mathrm{Sym}(\kk^{n+1})$ be a  polynomial ring  over $\kk = \overline{\kk}$.
All 2-determinantal ideals of $S$ of a fixed codimension $c$ are parametrized by points of the same Hilbert scheme $\Hilb^{p(\zeta)}(\P^n)$,
see  Section \ref{SectionPreliminaries}.
It is therefore natural to investigate 
 degenerations among 2-determinantal ideals: 
when does a given 2-determinantal ideal degenerate to another one?
In particular, for applications to blowup algebras, singularities, defining equations, etc., it is desirable to identify a set of 
``most degenerate'' 2-determinantal ideals in the Hilbert scheme.
There is a straightforward answer to the latter problem:
the revlex generic initial ideal of any 2-determinantal ideal is  of the form $(x_1, \ldots, x_c)^2$,
which is itself 2-determinantal, of nilpotent KW type.
However, 
this answer is not satisfactory,
because any such degeneration will usually fail to induce a degeneration of the blowup algebras.

We are interested in degenerations of 2-determinantal ideals $I=I_2(\bfM)$,
not of nilpotent KW-type,
 with the additional requirement that    $\codim \, I_1(\bfM)$ is constant throughout the degeneration.
Up to adjoining variables,
the problem reduces to the case where
$\codim \, I_1(\bfM)  = \dim_\kk[S]_1 = n+1$,
that is, where the entries of $\bfM$ generate the maximal ideal of $S$.
In terms of KW type, 
this  amounts to the fact
that the number of scroll blocks is equal to $d = n-c = \dim \V(I)$.

\begin{definition}\label{DefHcd}
Let $c,d\in \N$ and let $n = c+d$.
We denote by $\mcH_{c,d}\subseteq \Hilb^{p(\zeta)}(\P^n)$ the locus of 2-determinantal ideals in $S=\mathrm{Sym}(\kk^{n+1})$ of codimension $c$ that have $d$ scroll blocks in their Kronecker-Weierstrass normal form.
\end{definition}

Equivalently, $\mcH_{c,d}\subseteq \Hilb^{p(\zeta)}(\P^n)$ is the locus of 2-determinantal schemes that are not cones.
It is clear that 
$\mcH_{c,d}\ne \emptyset$ if and only if $0 \leq d \leq c+1$.
If $c \geq 2$, 
the closure  $\overline{\mcH_{c,d}}$ is 
a generically smooth  irreducible  component of $\Hilb^{p(\zeta)}(\P^n)$, 
see for instance \cite[Theorem 3.3]{KleppeMiroRoig}.
The case $c=1$ is not very interesting:
we have $d \leq 2$, 
$ \Hilb^{p(\zeta)}(\P^{d+1})$ is just the projective space of quadric hypersurfaces, and $\mcH_{1,d}$ is the locus of determinantal quadrics.

Let $\mathcal{L}_{(\lambda;\mu)} \subseteq \mcH_{c,d}$ denote the locus  of 2-determinantal ideals of KW type $(\lambda;\mu)$.
By Proposition \ref{PropExistenceConcatenation} and Definition \ref{DefKWNormalForm}, 
the subset  $\mathcal{L}_{(\lambda;\mu)}$ is  the image of a morphism $\mathcal{U}\times \mathrm{Aut}_\kk(S) \to \Hilb^{p(\zeta)}(\P^n)$, 
where $\mathcal{U}$ is the affine variety parametrizing  tuples of distinct eigenvalues
and $\mathrm{Aut}_\kk(S)$ is the group of $\kk$-algebra automorphisms of $S$, that is, 
$\mathrm{Aut}_\kk(S) \cong \GL_{n+1}(\kk)$.
It follows that each $\mathcal{L}_{(\lambda;\mu)}$  is irreducible, and that it is a constructible subset of $\Hilb^{p(\zeta)}(\P^n)$, i.e., it is a disjoint union of locally closed subsets. 
Moreover, by Theorem \ref{ThmPrimaryDecomposition} and Definition \ref{DefKWtype},
the space $\mcH_{c,d}$ is equal to the disjoint union of all the 
 $\mathcal{L}_{(\lambda;\mu)}$,
 in other words,
 the loci  $\mathcal{L}_{(\lambda;\mu)}$ form a stratification of $\mcH_{c,d}$.
In order to answer the questions stated  at the beginning of this section, 
 we will study  the poset  defined  by the relation $\mathcal{L}_{(\lambda;\mu)}\subseteq \overline{\mathcal{L}_{(\tau;\rho)}}$.

A classical theorem of Harris gives a complete solution 
for the case of scrolls.

\begin{thm}[{\cite[Section 3']{Harris}}]\label{TheoremHarrisPoset}
Let $I,I'\in \mcH_{c,d}$ be the ideals of two rational normal scrolls, with partitions 
$\lambda = (\lambda_1, \ldots, \lambda_d)$ and 
$\lambda' = (\lambda'_1, \ldots, \lambda'_{d})$,
respectively.
Then, $I$ degenerates to $I'$ if and only if 
and $\lambda$ is \emph{more balanced} than $\lambda'$, that is,
if and only if
$
\sum_{i=e}^d \lambda_i \geq \sum_{i=e}^d \lambda'_i
$
 for every
$
e = 1, \ldots, d.
$
\end{thm}

It follows that,
if we restrict to  scrolls, 
there is a unique most degenerate rational normal scroll
in each $\mcH_{c,d}$,
namely,
 the one with the least balanced partition $(c-d+2,1^{d-1})$.
 
Our main result in this section is the existence of a unique most degenerate  2-determinantal ideal in each $\mcH_{c,d}$, up to linear changes of coordinates.

\begin{thm}\label{TheoremUniqueMinimumHcd}
For every ideal $I\in \mcH_{c,d}$, 
there is a sequence of  degenerations  to an ideal of KW type $(1^d;c-d+1)$.
Moreover, the locus $\mathcal{L}_{(1^d;c-d+1)} $
is the unique closed stratum of $ \mcH_{c,d}$, and it is a single $\mathrm{Aut}_\kk(S)$-orbit.
\end{thm}

In order to prove Theorem \ref{TheoremUniqueMinimumHcd}, we  construct two basics degenerations.

\begin{lemma} \label{LemmaDegenerationPeelingScrollColumn} 
Let $\lambda$ be a partition, 
and let 
 $\lambda'$ be a partition obtained from $\lambda$ by replacing one part $\lambda_j>1$ with $\lambda_j-1$.
An ideal with KW type $(\lambda;\mu)$ degenerates 
to an ideal with KW type $(\lambda';\mu,1)$. 
Moreover, we have $\mathcal{L}_{(\lambda';\mu,1)}\subseteq \overline{\mathcal{L}_{(\lambda;\mu)}}$. 
\end{lemma}

\begin{proof}
Let $\bfM = 
(\,\bfS_1 \,|\, \cdots \,|\, \bfS_d \,|\, \bfJ_1 \,|\, \cdots \,|\, \bfJ_e )  $ be a concatenation of scroll blocks $\bfS_i$ of size $\lambda_i$ and Jordan blocks 
$\bfJ_i$ of size $\mu_i$ with distinct eigenvalues $\eps_i$.
For simplicity, we rename $p = \lambda_j$.
Consider the $j$-th scroll block
$$
\bfS_j = \begin{pmatrix}
x_1 & x_2 & \cdots & x_{p-1} & x_{p} \\
x_2 & x_3 & \cdots & x_{p} & x_{p+1} 
\end{pmatrix}.
$$
Let $\eps \in \kk$ be such that $\eps \ne 0$ and $\eps^{-1} \ne \eps_i $ for all $i$.
Set
$$
\bfS_j(t) = \begin{pmatrix}
x_1 & x_2 & \cdots & x_{p-1} & t x_{p} +   \eps x_{p+1}\\
x_2 & x_3 & \cdots & x_{p} & x_{p+1} 
\end{pmatrix}
$$
and
$\bfM(t) = 
(\,\bfS_1 \,|\, \cdots \,|\, \bfS_j(t) \,|\, \cdots \,|\, \bfS_d \,|\, \bfJ_1 \,|\, \cdots \,|\, \bfJ_e )$.
We claim that the family of ideals 
$ I_2(\bfM(t)) \subseteq S[t]$ is flat, and  that the
KW type of $I_2(\bfM(t))$ is
$(\lambda;\mu)$
for $t \ne 0$ 
and 
$(\lambda';\mu,1)$
 for $t=0$.

For $t=0$, the change of coordinates $x_{p+1}\mapsto \eps^{-1} x_{p+1}$  transforms  
$\bfS_j(0)$ into a concatenation of a scroll block of size $p -1$ and a Jordan block of size 1 with eigenvalue $\eps^{-1}$. Thus, $I_2(\bfM(0))$ has the desired KW type $(\lambda';\mu,1)$.

For $t \ne 0$, 
the  change of coordinates $x_i \mapsto \frac{1}{t}(x_i - \eps x_{i+1})$ 
for $i = 1, \ldots, p$ transforms  $\bfS_j(t)$ to 
\begin{equation}\label{EqScrollBlockColumnOps}
\begin{pmatrix}
\frac{1}{t}(x_1 - \eps x_2) & \frac{1}{t}(x_2 - \eps x_3)& \cdots &  \frac{1}{t}(x_{{p}-1} - \eps x_{p})		& x_{p}  \\
\frac{1}{t}(x_2 - \eps x_3) & \frac{1}{t}(x_3 - \eps x_4) & \cdots & \frac{1}{t}(x_{p} - \eps x_{{p}+1})		 & x_{{p}+1} 
\end{pmatrix}.
\end{equation}
Denoting by $\bfC_i$ the $i$-th column of \eqref{EqScrollBlockColumnOps},
we perform
the column operations $\bfC_i \mapsto t \bfC_i + \eps \bfC_{i+1}$
for  $i = p-1, p-2, \ldots, 1,$
transforming the block \eqref{EqScrollBlockColumnOps} into the original scroll block $\bfS_j$.
We conclude that $I_2(\bfM(t))$ has KW type 
$(\lambda;\mu)$ for all $t \ne 0$.

The flatness of the family $I_2(\bfM(t))$ follows from Lemma \ref{LemmaFlatnessHilbertFunction}:
all ideals in the family have the same  Hilbert function
since they are 2-determinantal ideals with the same codimension (Proposition \ref{PropKWNormalFormMaximalCodimension}).

The last statement of the lemma follows by the same argument.
Indeed, 
 given any 2-determinantal ideal $I \in \mathcal{L}_{(\lambda';\mu,1)}$,
we may assume, by Lemma \ref{LemmaTranslatingEigenvalues}, that all the eigenvalues in a  Kronecker-Weirstrass normal form are non-zero.
Now, the argument above shows that $I$ is a flat limit of ideals in $\mathcal{L}_{(\lambda;\mu)}$.
\end{proof}

\begin{lemma} \label{LemmaDegenerationTwoJordanBlocks} 
Let $\mu$ be a partition, 
and let 
 $\mu'$ be a partition obtained from $\mu$ by replacing two parts $\mu_h,\mu_k$ with one part $\mu_h+\mu_k$.
An ideal with KW type $(\lambda;\mu)$ degenerates 
to an ideal with KW type $(\lambda;\mu')$. 
Moreover,  we have $\mathcal{L}_{(\lambda;\mu')}\subseteq\overline{\mathcal{L}_{(\lambda;\mu)}}$.
\end{lemma}
\begin{proof}
Let $\bfM = 
(\,\bfS_1 \,|\, \cdots \,|\, \bfS_d \,|\, \bfJ_1 \,|\, \cdots \,|\, \bfJ_e )  $ be a concatenation of scroll blocks $\bfS_i$ of size $\lambda_i$ and Jordan blocks 
$\bfJ_i$ of size $\mu_i$ with distinct eigenvalues $\eps_i$.
By Lemma \ref{LemmaTranslatingEigenvalues},
 we may assume $\eps_h=0$, and for simplicity we rename $\eps = \eps_k,\, p = \mu_h,$ and $q = \mu_k$.
Consider the two Jordan blocks 
$$
\bfJ_h = \begin{pmatrix}
x_1 &  \cdots & x_{p-1} & x_{p} \\
x_2 &  \cdots & x_{p} & 0
\end{pmatrix} \,\,\, \text{and} \,\,\,\,
\bfJ_k = \begin{pmatrix}
y_1 &  \cdots & y_{q-1} & y_{q} \\
y_2 +\eps y_1 &  \cdots & y_{q} +\eps y_{q-1} & \eps y_{q}
\end{pmatrix}.
$$
We modify them setting
$$
\bfJ_h(t) =
 \begin{pmatrix}
x_1 & \cdots & x_{p-1} & x_{p} \\
x_2 &\cdots & x_{p} & (1-t) y_1
\end{pmatrix} \,\,\, \text{and} \,\,\,\,
\bfJ_k(t) = 
\begin{pmatrix}
y_1 &  \cdots & y_{q-1} & y_{q} \\
y_2 +t\eps y_1 & \cdots & y_{q} +t\eps y_{q-1} & t\eps y_{q}
\end{pmatrix}.
$$
Let
$\bfM(t) = 
(\,\bfS_1 \,|\, \cdots \,|\,  \bfS_d \,|\, \bfJ_1 \,|\, \cdots \,|\,
\bfJ_h(t) \,|\, \cdots \,|\,
\bfJ_k(t) \,|\, \cdots \,|\,
 \bfJ_e )$
 be the corresponding concatenation.
 As in Lemma \ref{LemmaDegenerationPeelingScrollColumn},
 we claim that $I_2(\bfM(t)) \subseteq S[t]$ is a flat family yielding the conclusion of the lemma.

For $t=0$, the ideal $I_2(\bfM(0))$ is clearly of KW type $(\lambda;\mu')$,
since the matrix $(\,\bfJ_h(0)\,|\,\bfJ_k(0)\,)$ is a single Jordan block with eigenvalue 0 and size $p+q$.

Now let  $t \ne 0$.
By repeatedly subtracting  multiples of the columns of $\bfJ_k(t)$ from the last column of $\bfJ_h(t)$,
we transform the matrix $\bfJ_h(t)$ into
$$
 \begin{pmatrix}
x_1 & \cdots & x_{p-1} & x_{p} + \sum_{i=1}^{q} \alpha_i y_i \\
x_2 &\cdots & x_{p} & 0
\end{pmatrix}
$$
for suitable  $\alpha_i \in \kk$.
Applying the change of coordinates $x_p \mapsto x_p - \sum_{i=1}^q \alpha_i y_q $ we obtain
$$
\begin{pmatrix}
x_1  & \cdots &x_{p-1} 	& x_{p}   \\
x_2  & \cdots & x_p- \sum_{i=1}^q \alpha_i y_q	 & 0
\end{pmatrix}.
$$
Again,
subtracting multiples of the columns of $\bfJ_k(t)$
 from the second last column of this matrix we get
$$
\begin{pmatrix}
x_1  & \cdots &x_{p-1} + \sum_{i=1}^q \alpha_i' y_i	 	& x_{p}  \\
x_2  & \cdots & x_p  & 0
\end{pmatrix}
$$
for suitable $\alpha_i' \in \kk$. 
Repeating this process of changing coordinates and column operations, we see that we can transform the  matrix $\bfJ_h(t)$ into
$$
\begin{pmatrix}
x_1 + \sum_{i=1}^q \alpha_i'' y_i	  & \cdots &x_{p-1} 	& x_{p}  \\
x_2  & \cdots & x_p  & 0
\end{pmatrix}. 
$$
A final change of coordinates $x_1 \mapsto x_1 - \sum_{i=1}^q \alpha_i''y_i$ 
restores the original block $\bfJ_h$.
We conclude that 
$I_2(\bfM(t))$ is of KW type $(\lambda;\mu)$ for 
 general $t \in \kk$, that is, for $t \in \kk$ such that $t \ne 0$ and $t\eps \ne \eps_i$ for all $i \ne h,k$.
Flatness follows as in Lemma \ref{LemmaDegenerationPeelingScrollColumn} by Lemma \ref{LemmaFlatnessHilbertFunction} and Proposition \ref{PropKWNormalFormMaximalCodimension}.

The last statement of the lemma also follows as in Lemma \ref{LemmaDegenerationPeelingScrollColumn}:
for any $I \in \mathcal{L}_{(\lambda;\mu')}$,
we may assume, by Lemma \ref{LemmaTranslatingEigenvalues}, that the eigenvalue corresponding to the Jordan block of size $\mu_h+\mu_k$ is 0,
and the argument above shows that $I$ is a flat limit of ideals in $\mathcal{L}_{(\lambda;\mu)}$.
\end{proof}

\begin{proof}[Proof of Theorem \ref{TheoremUniqueMinimumHcd}]
Applying Lemmas \ref{LemmaDegenerationPeelingScrollColumn} and \ref{LemmaDegenerationTwoJordanBlocks}
repeatedly, it follows that for any ideal $I\in \mcH_{c,d}$ there is a sequence of degenerations 
to an  ideal of KW type $(1^d;c-d+1)$.
Consider the action of $\mathrm{Aut}_\kk(S)$ on $\Hilb^{p(\zeta)}(\P^n)$.
It restricts to an action on $\mcH_{c,d}$.
It follows by Lemma  \ref{LemmaTranslatingEigenvalues} that  $\mathcal{L}_{(1^d;c-d+1)}$ is a single orbit of 
$\mathrm{Aut}_\kk(S)$.
By Lemmas \ref{LemmaDegenerationPeelingScrollColumn} and \ref{LemmaDegenerationTwoJordanBlocks},
all orbits different from $\mathcal{L}_{(1^d;c-d+1)}$ are not closed in $\mcH_{c,d}$.
Since orbits of minimal dimension are closed, it follows that $\mathcal{L}_{(1^d;c-d+1)}$ is  closed in $\mcH_{c,d}$.
\end{proof}

We now discuss the main application of the results of this section to the study of blowup algebras.
In general, 
given a flat family of ideals, 
their Rees algebras or the special fiber rings need not form flat families.
However, this  turns out to be true for 2-determinantal ideals $I_2(\bfM)$, as long as the number $\codim\, I_1(\bfM)$ is constant along the family.

\begin{lemma}\label{LemmaConstantHilbertFunctionBlowups}
Let $I \in \mcH_{c,d}$.
The Hilbert functions of $\mcR(I)$ and $\mcF(I)$
depend only on $c,d$.
\end{lemma}
\begin{proof}
By \cite[Theorem 3.7]{BrunsConcaVarbaro},
 the graded Betti numbers of  each power $I^k$ are uniquely determined by $c,d,k$.
It follows that the Hilbert functions of $\mcR(I)$ and $\mcF(I)$ are uniquely determined by $c,d$.
\end{proof}

\begin{cor}\label{CorDegenerationBlowupAlgebras}
Let $I,J \in \mcH_{c,d}$ be ideals such that 
 $I$ degenerates to $J$. 
 Then, $\mcR(I)$ degenerates to $\mcR(J)$ and $\mcF(I)$ degenerates to $\mcF(J)$.
\end{cor}
\begin{proof}
A one-parameter family of ideals, as defined in Section \ref{SectionPreliminaries}, determines a one-parameter family of their Rees algebras,  likewise for the special fiber rings. 
The claim follows from  Lemmas \ref{LemmaConstantHilbertFunctionBlowups} and \ref{LemmaFlatnessHilbertFunction}.
\end{proof}

While Theorem \ref{TheoremUniqueMinimumHcd} and Corollary \ref{CorDegenerationBlowupAlgebras} are sufficient for 
our investigation of the blowup algebras,
it is natural to ask for the full poset of degenerations
among KW types  in $\mcH_{c,d}$,
that is, the complete extension of Theorem \ref{TheoremHarrisPoset} to all 2-determinantal ideals.
It is possible to show that this  poset is precisely the one generated by the degenerations of Theorem \ref{TheoremHarrisPoset}, Lemma \ref{LemmaDegenerationPeelingScrollColumn}, and Lemma \ref{LemmaDegenerationTwoJordanBlocks}.
In other words, 
an ideal
of KW type
$(\lambda_1, \ldots, \lambda_d; \mu_1, \ldots, \mu_p)$ 
degenerates to one of KW type
$(\lambda'_1, \ldots, \lambda'_d; \mu'_1, \ldots, \mu'_q)$
if and only if there exist $a_1, \ldots, a_d\in \N$ such that
\begin{itemize}
\item  $a_i <\lambda_i$ for all $i=1,\ldots,d$ and $a =\sum_{i=1}^d a_i = \sum_{i=1}^d  \lambda_i -\sum_{i=1}^d  \lambda'_i$,
\item the partition $\lambda'' $ obtained by sorting $(\lambda_1-a_1, \ldots, \lambda_d-a_d)$
is more balanced than $\lambda'$,
\item the  partition $\mu'' = (\mu_1, \ldots, \mu_p, 1^a)$
is a refinement of $\mu'$, in the sense that there exists a set partition $\mathcal{P}_1, \ldots, \mathcal{P}_q$ of $\{1, \ldots, p+a\}$ such that $\mu'_j = \sum_{i \in \mathcal{P}_j} \mu''_i$ for all $j=1, \ldots, q$.
\end{itemize}
Again, the condition is sufficient by 
Theorem \ref{TheoremHarrisPoset} and Lemmas \ref{LemmaDegenerationPeelingScrollColumn}, \ref{LemmaDegenerationTwoJordanBlocks},
whereas a formal proof of the necessity would take us too far away from  the scope of this paper and therefore we omit it.

We end the section by illustrating the case of $\mcH_{6,3}$, i.e., the poset of degenerations of a $3$-dimensional balanced rational normal scroll of degree $7$.
\begin{center}
\begin{tikzpicture}
    \node (a1) at (0,0) {$(3,2^2;\emptyset)$};
    \node (a2) at (1.5,-1.5) {$(3^2,1;\emptyset)$};
    \node (a3) at (1.5,-3) {$(4,2,1;\emptyset)$};
    \node (a4) at (1.5,-4.5) {$(5,1^2;\emptyset)$};
    \node (b1) at (-1.5,-1.5)  {$(2^3;1)$};
    \node (b2) at (-1.5,-4.5)  {$(3,2,1;1)$};
    \node (b3) at (1.5,-6)  {$(4,1^2;1)$};
    \node (c1) at  (-1.5,-6) {$(2^2,1;1^2)$};
    \node (c2) at (-3,-7.5) {$(2^2,1;2)$};
    \node (c3) at (0,-7.5) {$(3,1^2;1^2)$};
    \node (c4) at (3,-9) {$(3,1^2;2)$};
    \node (d1) at  (-1.5,-9) {$(2,1^2;1^3)$};
    \node (d2) at (-3,-10.5) {$(2,1^2;2,1)$};
    \node (d3) at (-4.5,-12) {$(2,1^2;3)$};
    \node (e1) at  (1.5,-10.5) {$(1^3;1^4)$};
    \node (e2) at (-1.5,-12) {$(1^3;2,1^2)$};
    \node (e3) at (-3,-13.5) {$(1^3;3,1)$};
    \node (e4) at (0,-13.5) {$(1^3;2^2)$};
    \node (e5) at (-1.5,-15) {$(1^3;4)$};     
   \draw [thick] (a1) -- (a2) -- (a3) -- (a4) -- (b3) -- (c3) -- (d1) -- (e1) -- (e2) -- (e4) -- (e5);
   \draw [thick] (a1) -- (b1) -- (b2) -- (c1) -- (c2) -- (d2) -- (d3) -- (e3) -- (e5);
   \draw [thick] (a3) -- (b2);
   \draw [thick] (b2) -- (b3);
   \draw [thick]  (c1) -- (c3);
   \draw [thick]  (c3) -- (c4);
   \draw [thick]  (c4) -- (d2);   
   \draw [thick]  (d1) -- (d2);   
   \draw [thick]  (d2) -- (e2);  
   \draw [thick]  (e2) -- (e3);
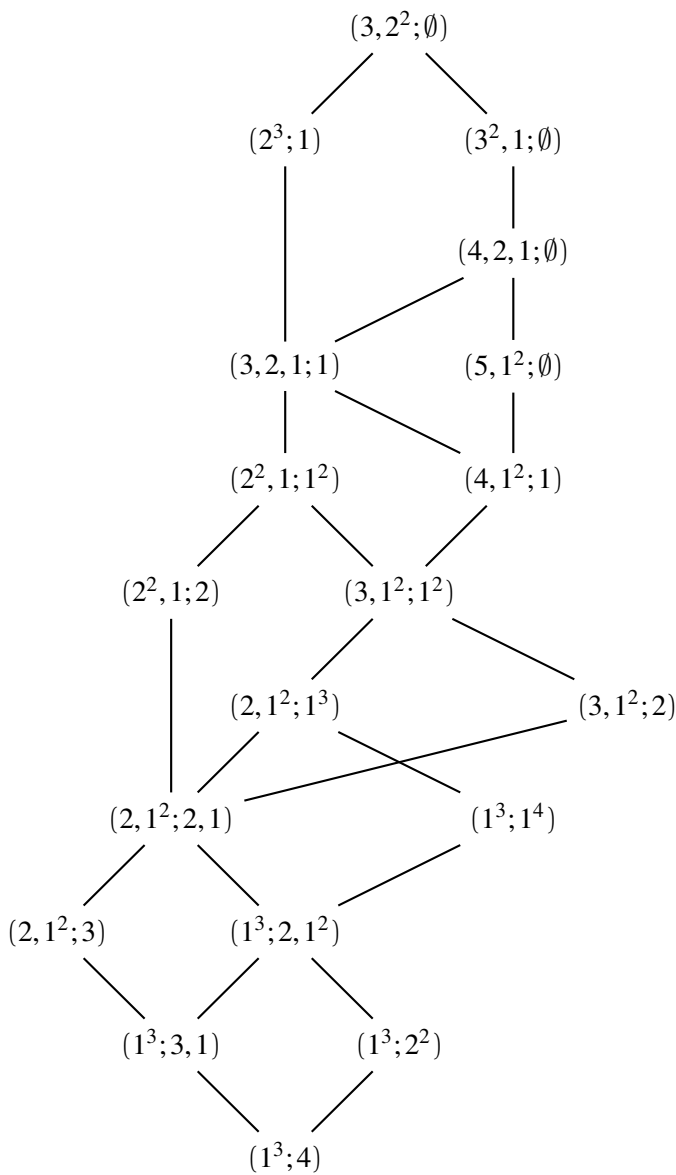
   \node at (0,-16) {Poset of degenerations of a three-dimensional rational normal scroll of degree 7.};
\end{tikzpicture}
\end{center}

\section{Blowup algebras of the most special determinantal ideals}\label{SectionBlowupAlgebras}

The results of Section \ref{SectionDegenerations} reduce the proof of  Theorem \ref{TheoremMain}
from arbitrary 2-determinantal ideals  to 
 a single ideal in each $\mcH_{c,d}$,
  of KW type  $(1^d;c-d+1)$. 
Our next goal is 
 completing this task.
In this section, we set up the notation for the blowup algebras, and introduce  four classes of polynomial relations and a monomial order for them.

Since we are only concerned with one KW type from now on, we fix a more convenient notation.

\begin{notation}\label{NotationBlowupSpecial}
Let   $d, e \in \N$ be non-negative integers such that $d+e\geq 3$, and 
let $c = d+e-1$.
Consider the following matrix and determinantal ideal
$$
\bfL=
\begin{pmatrix}
y_{1,1} & y_{1,2} & \cdots & y_{1,d} & x_1 & x_2 & \cdots & x_{e-1} & x_{e}   \\
y_{2,1} & y_{2,2} & \cdots & y_{2,d} &x_2 & x_3 & \cdots & x_{e} & 0 
\end{pmatrix} = (\ell_{i,j}) \quad \text{and} \quad
L=I_2(\bfL)
$$
in the polynomial ring $S = \kk[y_{1,1}, \ldots, y_{2,d}, x_1, \ldots, x_e]$.
Here, $\ell_{i,j}$ denotes the entry of $\bfL$ in position $(i,j)$.
\end{notation}

Thus,  $\bfL$ is
a $2\times(c+1)$ matrix
 in Kronecker-Weierstrass normal form and $L$ is a 2-determinantal ideal with  KW type $(1^d;e)$ and codimension $c$.

\subsection{Relations of the blowup algebras}

For each $\al \in [c+1]$ let  $C_\alpha$ denote the $\alpha$-th column of $\bfL$.
Let $\Delta_{\alpha,\beta}=\det(C_\al,C_\be)\in S$ be the $2\times 2 $ minor of $\bfL$
determined by two columns $\al < \be$.
The Rees algebra 
$$
\mcR(L) = S[\Delta_{\al,\be}\tau \, : \, 1 \leq \alpha < \beta \leq c+1] \subseteq S[\tau]
$$ is bigraded  
by  $\deg(y_{i,j})=\deg(x_h)=(1,0)$ and $\deg(\tau)=(-2,1)$.
In particular, $\mcR(L)$ is a standard bigraded $\kk$-algebra, generated in bidegrees $(1,0)$ and $(0,1)$.
As in Section \ref{SectionPreliminaries},
we introduce new variables $T_{\al,\be}$ corresponding to the generators of $L$, and define a polynomial presentation
\begin{equation}\label{EqPresentationRees}
\pi_\mcR \, : \, P_\mcR := S[T_{\al,\be}\, : \, 1 \leq \alpha < \beta \leq c+1] \rightarrow \mcR(L) 
\qquad \text{ by } \quad
\pi_\mcR ( T_{\al,\be}) = \Delta_{\al,\be}\tau.
\end{equation}
The ring $P_\mcR$ is bigraded by  $\deg(y_{i,j})=\deg(x_h)=(1,0)$, $\deg(T_{\al,\be})=(0,1)$, and the map $\pi_\mcR$ is homogeneous and surjective.
The special fiber ring  $\mcF(L)$ is the subring of $\mcR(L)$ concentrated in bidegrees $(0,\ast)$,
and we identify it with the subalgebra of $S$ generated by the minors of $\bfL$
$$
\mcF(L) =\kk[\Delta_{\al,\be}\, : \, 1 \leq \alpha < \beta \leq c+1] \subseteq S.
$$
Restricting the presentation \eqref{EqPresentationRees} to
bidegrees $(0,\ast)$ we obtain 
a polynomial presentation
$$
\pi_\mcF \, : \, P_\mcF := \kk[T_{\al,\be}\, : \, 1 \leq \alpha < \beta \leq c+1] \rightarrow \mcF(L) 
\qquad \text{ by } \quad
\pi_\mcF ( T_{\al,\be}) = \Delta_{\al,\be}.
$$
The defining ideals of $\mcR(L)$ and $\mcF(L)$
are $\ker(\pi_\mcR)$ and $\ker(\pi_\mcF)$, respectively.
The relations of $\mcF(L)$ are exactly 
the relations of $\mcR(L)$ of bidegrees $(0,\ast)$, that is,  $\ker(\pi_\mcF)=[\ker(\pi_\mcR)]_{(0,\ast)}$.
We now describe four classes of relations of $\mcR(L)$ and $\mcF(L)$.

 For every 3-subset $\{\al,\be,\ga\} \in {[c+1] \choose 4}$ we have the {\bf upper Eagon-Northcott relation} 
\begin{equation}\label{EqUpperEagonNorthcott}
\UEN_{\alpha, \beta, \gamma} =  
\ell_{1,\al}T_{\be,\ga} - \ell_{1,\be}T_{\al,\ga}+\ell_{1,\ga}T_{\al,\be}
\end{equation}
and the 
{\bf lower Eagon-Northcott relation} 
\begin{equation}\label{EqLowerEagonNorthcott}
\LEN_{\alpha, \beta, \gamma} =  
\ell_{2,\al}T_{\be,\ga} - \ell_{2,\be}T_{\al,\ga}+\ell_{2,\ga}T_{\al,\be}.
\end{equation}
They can be obtained from the vanishing of the determinant obtained from the $2\times 3$ matrix
$(C_\al, C_\be, C_\ga)$  by duplicating the upper and lower row, respectively:
$$
\det 
\begin{pmatrix}
\ell_{1,\al} & \ell_{1,\be} & \ell_{1,\ga} \\
C_\alpha & C_\beta & C_\gamma 
\end{pmatrix}=
\det 
\begin{pmatrix}
C_\alpha & C_\beta & C_\gamma \\
\ell_{2,\al} & \ell_{2,\be} & \ell_{2,\ga} 
\end{pmatrix}=0.
$$

 For every $\{\al,\be,\ga,\de\} \in {[c+1] \choose 3}$ we have the {\bf Plucker relation} 
\begin{equation}\label{EqPluckerRelation}
\PLU_{\alpha, \beta, \gamma, \delta} =  T_{\alpha,\beta} T_{\gamma,\delta} - T_{\alpha,\gamma} T_{\beta,\delta} +T_{\alpha,\delta} T_{\beta,\gamma}.
\end{equation}
When $\mathrm{char}(\kk)\ne 2$, it  can be obtained from  the vanishing of the determinant obtained by duplicating the entire 
$2\times 4$ matrix
$(C_\al, C_\be, C_\ga, C_\de)$
$$
\det 
\begin{pmatrix}
C_\alpha & C_\beta & C_\gamma & C_\delta \\
C_\alpha & C_\beta & C_\gamma & C_\delta 
\end{pmatrix}=0
$$

Finally, for every $\{\al,\be,\ga,\de\} \in {\{d+1,\ldots, c+1\}\choose 4}$
 we have a polynomial relation 
$\LAP_{\alpha, \beta, \gamma, \delta}$,  which we call a {\bf Laplace relation}, obtained as follows.
If  $\delta \leq c$, the following $4\times 4 $ determinant vanishes
$$
\det 
\begin{pmatrix}
C_\alpha & C_\beta & C_\gamma & C_\delta \\
C_{\alpha+1} & C_{\beta+1} & C_{\gamma+1} & C_{\delta+1}
\end{pmatrix}=0,
$$
since the second and third  rows coincide.
This yields the relation
\begin{equation}\label{EqLaplaceNoLastColumn}
\begin{split}
\LAP_{\alpha, \beta, \gamma, \delta} =
T_{\alpha,\beta} T_{\gamma+1,\delta+1} - T_{\alpha,\gamma} T_{\beta+1,\delta+1}+T_{\alpha,\delta} T_{\beta+1,\gamma+1}\\+
T_{\beta,\gamma} T_{\alpha+1,\delta+1} - T_{\beta,\delta} T_{\alpha+1,\gamma+1}+T_{\gamma,\delta} T_{\alpha+1,\beta+1}.
\end{split}
\end{equation}
If $\delta = c+1$,
the following $4\times 4 $ determinant vanishes
$$
\det 
\begin{pmatrix}
C_\alpha & C_\beta & C_\gamma & C_{c+1} \\
C_{\alpha+1} & C_{\beta+1} & C_{\gamma+1} & 0
\end{pmatrix}=0,
$$
yielding the relation 
\begin{equation}\label{EqLaplaceLastColumn}
\LAP_{\alpha, \beta, \gamma, c+1} =
T_{\alpha,c+1} T_{\beta+1,\gamma+1}
- T_{\beta,c+1} T_{\alpha+1,\gamma+1}+T_{\gamma,c+1} T_{\alpha+1,\beta+1}.
\end{equation}

\noindent
Eagon-Northcott relations have bidegree $(1,1)$,  while  Plucker and Laplace relations have bidegree $(0,2)$.

\begin{remark}
The first three classes of relations are  well-known, 
and they exist in general for all ideals of maximal minors.
The Eagon-Northcott relations are  the first syzygies in the Eagon-Northcott complex.
The Plucker relations are the defining relations of the Grassmannian variety.
The relations of the fourth class are less well-known, as  they arise from the special structure of the matrix $\bfL$; however, similar relations  have appeared in \cite{CHV} and \cite{Sammartano} in the case of rational normal scrolls.
\end{remark}

\subsection{The monomial order}\label{SubsectionMonomialOrder}
We now introduce a monomial order $\preceq$ on the  polynomial rings 
$P_\mcR = \kk[y_{i,j},x_h,T_{\al,\be}]$
and, by restriction, $P_\mcF = \kk[T_{\al,\be}]$.
This  order, while apparently convoluted, 
has the benefit of determining two well-behaved initial complexes for $\mcR(L)$ and $\mcF(L)$, specifically, two flag complexes with good homological and combinatorial properties, as we will see in  the next sections.
The  definition of $\preceq$ is not particularly important: what matters are the  initial terms it picks from the polynomial relations described above.

The  order $\preceq$ is constructed by defining a hierarchy of rules for breaking ties among monomials.
\begin{enumerate}
\item Given  two monomials $\bfu, \bfv \in P_\mcR$
we factor  $\bfu = \bfu' \bfu'', \bfv= \bfv'\bfv''$
where $\bfu',\bfv' \in \kk[y_{2,1},\ldots, y_{2,d}]$ and 
$\bfu'',\bfv'' \in \kk[y_{1,1},\ldots, y_{1,d}, x_1, \ldots, x_e,
T_{\al,\be}\, : \, 1 \leq \alpha < \beta \leq c+1]$,
and declare $\bfu \succeq \bfv$ if $\bfu' > \bfv'$ with respect to the lexicographic order induced by $y_{2,1} > y_{2,2} > \cdots > y_{2,d}$.
In other words, in the first step we order monomials lexicographically considering only the variables $y_{2,1},\ldots, y_{2,d}$.
\item 
Let   $ \mathbb{N}^{c+1}$ be the free abelian monoid with  basis $\bfe_1, \ldots, \bfe_{c+1}$,
equipped with the graded reverse-lexicographic order  induced by $\bfe_{c+1} > \bfe_{c} > \cdots > \bfe_1$.
We define a multigrading on $P_\mcR$ by
$$
\qquad
\mdeg_1(T_{\alpha,\beta})=\mathbf{e}_\alpha+\mathbf{e}_\beta,
\quad 
\mdeg_1(y_{1,\al})=\bfe_\al, 
\quad
\mdeg_1(y_{2,\al})=\mathbf{0}, 
\quad
\mdeg_1(x_\al) = \bfe_{\al+d}.
$$ 
In other words, $\mdeg_1(\cdot)$  keeps track of columns in the matrix $\bfL$; for the variable $x_\al$, it selects the column on which it appears in the first row.
In the second step, we compare monomials according to their multidegree in $\N^{c+1}$.
\item 
Consider again   $ \mathbb{N}^{c+1}$, now 
equipped with the graded reverse-lexicographic order
induced by $\bfe_{1} > \bfe_{2} > \cdots > \bfe_{c+1}$.
We define another multigrading on $P_\mcR$ by
$
\mdeg_2(T_{\alpha,\beta})=\mathbf{e}_\alpha   
$
and
$\mdeg_2(y_{1,\al})=
\mdeg_2(y_{2,\al})=
\mdeg_2(x_\al) = \mathbf{0},
$ 
and we compare monomials according to this multigrading.
 \item In the final step, we compare monomials using the lexicographic order on $P_\mcR$ induced by
$$
y_{2,1} >  \cdots > y_{2,d} > 
y_{1,1} >  \cdots > y_{1,d} > x_1 > T_{\alpha,\beta} > x_2 > \cdots > x_e \quad \text{for all} \,\,\, \al, \be
$$
 where the  $T_{\alpha,\beta}$ are ordered by $
 T_{1,1} > T_{1,2} > \cdots > T_{1,c+1} > T_{2,1} > \cdots > T_{c,c+1}.
$
\end{enumerate}

\begin{prop}\label{PropLeadingMonomials}
The leading monomials of the polynomials \eqref{EqUpperEagonNorthcott},
\eqref{EqLowerEagonNorthcott},
\eqref{EqPluckerRelation},
\eqref{EqLaplaceNoLastColumn},
\eqref{EqLaplaceLastColumn}
are
\begin{eqnarray}
\LM(\UEN_{\al,\be,\ga}) &=& 
\begin{cases} 
\ell_{1,\ga} T_{\al,\be}= x_{\ga-d}T_{\al,\be} & \text{if } \be \geq d+2, \\
\ell_{1,\be} T_{\al,\ga}= x_1 T_{\al,\ga} & \text{if } \be = d+1,\\
\ell_{1,\be} T_{\al,\ga}= y_{1,\be} T_{\al,\ga} & \text{if } \be \leq d,\\
\end{cases}\\
\LM(\LEN_{\al,\be,\ga}) &=& 
\begin{cases} 
\ell_{2,\al} T_{\be,\ga}= y_{2,\al} T_{\be,\ga} & \text{if } \al \leq d,\\
\ell_{2,\al} T_{\be,\ga}= x_{\al+1-d} T_{\be,\ga} & \text{if } \al \geq d+1,\\
\end{cases}\\
\label{EqLMPlucker}
\LM(\PLU_{\al,\be,\ga,\de}) &=& T_{\al,\ga}T_{\be,\de}, \,\,\, \text{and}\\
\label{EqLMLaplace}
\LM(\LAP_{\al,\be,\ga,\de}) &=& T_{\al+1,\be+1}T_{\ga,\de}.
\end{eqnarray}
\end{prop}
\begin{proof}
In order to verify the statements, we need to compare the terms in the support of each polynomial following the four steps in the construction of $\preceq$ until we obtain a unique highest monomial.

Step (1) is only relevant for polynomials containing some $y_{2,i}$,
that is, for $\LEN_{\al,\be,\ga}$ with $\al \leq d$, 
and for them the unique highest term is the desired $y_{2,\al}T_{\be,\ga}$.

In step (2), we note that Plucker and upper Eagon-Northcott relations are homogeneous with respect to $\mdeg_1(\cdot)$,
so this step has no effect on them,
while the remaining relations are not homogeneous.
In both cases, $\delta \leq c $ and $\delta = c+1$, the unique highest monomial in $\LAP_{\alpha, \beta, \gamma, \delta}$ is
 $
T_{\alpha+1,\beta+1} T_{\gamma,\delta} 
$ as desired.
Likewise, 
for each $\LEN_{\alpha, \beta, \gamma}$ with $\al \geq d+1$ we  see that $x_{\al+1-d} T_{\be,\ga} $ is the unique highest  monomial.

In step (3), neither the Plucker nor the upper Eagon-Northcott relations are homogeneous with respect to $\mdeg_2(\cdot)$.
The two (equally) highest monomials  in 
$\PLU_{\alpha, \beta, \gamma, \delta}$ are
$T_{\alpha,\gamma} T_{\beta,\delta}$ and  $T_{\alpha,\delta} T_{\beta,\gamma}$,
while 
the two (equally) highest monomials in 
$\UEN_{\alpha, \beta, \gamma}$  are
$\ell_{1,\be}T_{\alpha,\gamma}$ and  $\ell_{1,\ga}T_{\alpha,\beta} $.

The last remaining ties are broken in step (4).
The highest variable in $\{T_{\alpha,\gamma} T_{\beta,\delta},T_{\alpha,\delta} T_{\beta,\gamma}\}$ is $T_{\al,\ga}$, 
so $\LM(\PLU_{\alpha, \beta, \gamma, \delta}) = T_{\alpha,\gamma} T_{\beta,\delta}$ by the lexicographic order.
The highest variable
in 
$\{\ell_{1,\be}T_{\alpha,\gamma},\ell_{1,\ga}T_{\alpha,\beta}\} $
is $\ell_{1,\be}$ if $\be \leq d+1$, 
 $T_{\alpha,\beta}$ if $\be \geq d+2$,
so we get the claimed leading monomial for $\UEN_{\al,\be,\ga}$.
\end{proof}

\section{Equations of the special fiber ring}\label{SectionSpecialFiber}

In this section, we prove the following theorem:

\begin{thm}\label{TheoremSpecialFiber}
The Plucker and Laplace relations 
\eqref{EqPluckerRelation},
\eqref{EqLaplaceNoLastColumn},
\eqref{EqLaplaceLastColumn}
 form a Gr\"obner basis of the defining ideal $\ker(\pi_\mcF)$ of the special fiber ring $\mcF(L)$
 with respect to $\preceq$.
\end{thm}

The proof strategy   is very similar to that of  \cite[Section 3]{Sammartano},
based on the enumerative analysis of the initial simplicial complex of $\mcF(L)$.
The analysis is  substantially trimmer   here, since we only consider the KW type $(1^d;e)$,
whereas in 
\cite{Sammartano} a different  complex is needed for every scroll partition.
Rather than giving a completely self-contained treatment,  
we present the full analysis but refer to \cite{Sammartano} 
for the combinatorial   calculations of Proposition \ref{PropCountFacets}.
The (new) case $d=0$ requires the additional computation of the multiplicity of $\mcF(L)$, which we perform in Appendix  \ref{AppendixMultiplicities}.

We refer to \cite{HerzogHibi} for background on monomial ideals, simplicial complexes, and the Stanley-Reisner correspondence.

Let $\Delta_\mcF$ be the flag simplicial complex defined by the squarefree quadratic monomials \eqref{EqLMPlucker} and \eqref{EqLMLaplace}, that is, the complex whose Stanley-Reisner ideal is
$$
\mcI(\Delta_\mcF)=
\Big(
T_{\al,\ga} T_{\be,\de}\,\, \big|\, 1\leq \alpha  < \be < \ga < \de \leq c+1\Big)+
\Big(
T_{\al,\be} T_{\ga,\de}\, \,\big|\, d +2\leq \alpha  < \be \leq \ga < \de \leq c+1\Big).
$$
Let   $\mcV  = \big\{ (\al,\be) \, \mid \{\al,\be\} \in {[c+1] \choose 2}	\big\}$ denote the vertex set of $\Delta_\mcF$.
We identify vertices  with open intervals of the real line $\mathbb{R}$, and   use   concepts such as inclusion,   length, and intersection.
Translating  the two classes of generators of $\mcI(\Delta_\mcF)$,  a subset $F \subseteq \mcV$ is a face of $\Delta_\mcF$ if and only if two conditions hold:
\begin{enumerate}
\item[($\dagger$)] Any two intervals  in $F$ are  \emph{non-crossing}: either they are
 disjoint or one is  contained in the other.
\item[($\ddagger$)] $F$ contains no two intervals  that are disjoint and 
contained in $(d+2,c+1)$.
\end{enumerate}
The relation of inclusion  between intervals 
makes the vertex set into a poset $(\mcV,\subseteq)$.
If a subset $F \subseteq \mcV$
satisfies $(\dagger)$,
then  the Hasse diagram of the poset $(F,\subseteq)$   is a graph without cycles, hence a union of (ordered) trees.
We use tree-related terminology such as root, children, leaves, and internal nodes.

When Laplace relations are absent, i.e. when $e \leq 3$,
$\Delta_\mcF$ is the \emph{non-crossing} complex of the Plucker algebra \cite{PPS},
and the numerical data  of its facets is easy to determine:

\begin{prop}\label{PropCaseE3}
Assume $e \leq 3$.
A subset $F \subseteq \mcV$ is a facet of $\Delta_\mcF$
if and only if the Hasse diagram of the poset $(F,\subseteq)$ 
satisfies the following conditions:
\begin{itemize}
\item $F$ is a  binary tree with root $(1,c+1)$;
\item the children of an internal node  $(\al,\be)\in F$  
 are of the form $(\al, \ga), (\ga, \be)$ for some $\al < \ga < \be$;
\item  the leaves of $F$ are  $(1,2),(2,3),\dots,(c,c+1)$.
\end{itemize}
Every facet contains $2c-1$ intervals,  and the number of facets is 
$ {2c-2 \choose c-1} -{2c-2 \choose c}$.
\end{prop}
\begin{proof}
Note that $(\ddagger)$ is vacuous when $e \leq 3$.
The  conditions  follow quickly from the fact that  a facet is a subset $F\subseteq \mcV$ that satisfies  $(\dagger)$ and  is maximal with respect to this property.
Thus, facets  correspond to full binary trees on $c$ leaves.
It is well-known that 
 such tree have $2c-1$ nodes,
 and that the number of such trees is the  Catalan   number  $C_{c-1} = {2c-2 \choose c-1} -{2c-2 \choose c-2}$.
\end{proof} 

When Laplace relations are present, and   $F \subseteq \mcV$ is a subset  satisfying $(\dagger)$, then $(\ddagger)$ holds  if and only if it holds  for the leaves of $F$.
This yields an analogous description of the facets of $\Delta_\mcF$.

\begin{prop}\label{PropCaseE4}
Assume $e \geq 4$.
A subset $F \subseteq \mcV$ is a facet of $\Delta_\mcF$
if and only if the Hasse diagram of the poset $(F,\subseteq)$ 
satisfies the following conditions:
\begin{itemize}
\item $F$ is a binary tree with root $(1,c+1)$;
\item if $(\al,\be)\in F$ has two children, then 
they are of the form $(\al, \ga), (\ga, \be)$ for some $\al < \ga < \be$;
\item if $(\al,\be)\in F$ has one child, then 
it is either $(\al+1,\be)$ or $(\al,\be-1)$;
\item  the leaves of $F$ are  $(1,2),(2,3),\dots,(d+1,d+2)$ and  $(\ell,\ell+1)$ for one $\ell \geq d+2 $.
\end{itemize}
Every facet contains $c+d+1$ intervals.
\end{prop}

\begin{proof}
Once again, the statements follow easily since a facet is a maximal subset $F$ that satisfies  $(\dagger), (\ddagger)$; we omit the details.
Alternatively, see  the similar \cite[Proposition 3.8]{Sammartano}.
\end{proof}

\begin{prop}\label{PropCountFacets}
Assume $e \geq 4$. 
The number of facets of $\Delta_\mcF$ is
$
\sum_{h=1}^{c-d-1}{c+d \choose h+d} - (c-d-1){c+d\choose d}.$
\end{prop}
\begin{proof}
Let $\Sigma(\ell)$ denote the set of facets 
containing the leaf $(\ell, \ell+1)$,
where  $d+2 \leq \ell \leq c$.
Applying the reflection $n \mapsto c+2-n$ to the interval $[1,  c+1]$,
the sets $\Sigma(\ell)$ become of the form $\Sigma(\alpha, \beta_1, \gamma, \beta_2)$ as in \cite[Lemma 3.11]{Sammartano}.
More precisely,
$\Sigma(d+1)$ becomes  $\Sigma(c-d-2,d+2,0,0)$,
while $\Sigma(\ell)$ 
becomes  $\Sigma(c-\ell,1,\ell-d-2,d+1)$
for $d+2 \leq \ell \leq c$.
Applying  \cite[Lemma 3.11]{Sammartano}, we see that 
$\mathrm{Card}\big(\Sigma(\ell)\big) = {c+d \choose \ell-1} -{c+d \choose d} $
for all $ d+2 \leq \ell \leq c$.
Reindexing $h = \ell-1-d$ we obtain the desired formula.
\end{proof}

Theorem \ref{TheoremSpecialFiber}  is now a consequence of the following criterion.
It is a well-known application of the associativity formula for multiplicities, but we include a proof since we could not find one in the literature.
A simplicial complex is called \emph{pure} if all its facets have the same cardinality.

\begin{lemma}\label{LemmaCriterionInitialComplex}
Let $R$ be a polynomial ring, $\leq$  a monomial order,
and   $\mathcal{K} \subseteq R$   a homogeneous ideal.
   Let $\Delta$ be a pure simplicial complex whose Stanley-Reisner ideal $\mcI(\Delta) \subseteq R$ satisfies $\mcI(\Delta) \subseteq \mathrm{in}_\leq(\mathcal{K})$.
Suppose that  the rings 
$R/\mathcal{K}$ and $R/\mcI(\Delta)$ have the same Krull dimension and multiplicity. 
Then, $\mcI(\Delta) = \mathrm{in}_\leq(\mathcal{K})$.
\end{lemma}
\begin{proof}
Denote $\mathcal{J}= \mathrm{in}_\leq(\mathcal{K})$.
Then, $R/\mathcal{J}$ and $R/\mcI(\Delta)$ also have the same Krull dimension and multiplicity. 
Since   $\mcI(\Delta)\subseteq \mathcal{J}$,
 the equality
$\mcI(\Delta)= \mathcal{J}$ 
can be checked locally at the associated primes of the smaller ideal $\mcI(\Delta)$.
Since all  facets of $\Delta$ have the same size, 
all associated primes of $\mcI(\Delta)$ have the same height.
Since $\mathcal{J}$ and $\mcI(\Delta)$ have equal  height,
the associated primes of $\mathcal{J}$ of minimal height are also associated primes of $\mcI(\Delta)$.
Finally,
since  $R/\mathcal{J}$ and $R/\mcI(\Delta)$ have the same multiplicity,
 the associativity formula for multiplicities
\cite[Exercise 12.11.e]{Eisenbud}
implies that 
$\mcI(\Delta)_{\mathcal{P}}= \mathcal{J}_{\mathcal{P}}$  for all $\mathcal{P}\in \mathrm{Ass}(\mcI(\Delta))$.
\end{proof}

\begin{proof}[Proof of Theorem \ref{TheoremSpecialFiber}]
The special fiber ring $\mcF(L)$ and the Stanley-Reisner ring $P_\mcF/\mcI(\Delta_\mcF)$ have the same dimension and multiplicity.
For $P_\mcF/\mcI(\Delta_\mcF)$, these invariants are found in Propositions \ref{PropCaseE3}, \ref{PropCaseE4}, \ref{PropCountFacets},
where it is also shown that $\Delta_\mcF$ is pure.
For  $\mcF(L)$,
they are found in Propositions \ref{PropositionInvariantsFiberD0}  and \ref{PropositionInvariantsFiberD1}.
Moreover,   $\mcI(\Delta_\mcF)\subseteq \mathrm{in}_\preceq(\ker(\pi_\mcF))$ by Proposition \ref{PropLeadingMonomials}.
Now, apply Lemma \ref{LemmaCriterionInitialComplex} to conclude.
\end{proof}

\section{Equations of the Rees algebra}\label{SectionReesAlgebra}

In this section, we prove the following theorem:

\begin{thm}\label{TheoremReesAlgebraGB}
The Eagon-Northcott,
Plucker and Laplace relations \eqref{EqUpperEagonNorthcott},
\eqref{EqLowerEagonNorthcott},
\eqref{EqPluckerRelation},
\eqref{EqLaplaceNoLastColumn},
\eqref{EqLaplaceLastColumn}
 form a Gr\"obner basis of the defining ideal $\ker(\pi_\mcR)$ of the Rees algebra $\mcR(L)$
 with respect to  $\preceq$.
\end{thm}

The strategy we adopt here is different from the one of Section \ref{SectionSpecialFiber}, for two reasons:
the initial  complex of $\mcR(L)$ is considerably more convoluted than that of $\mcF(L)$,
and 
 a formula for the multiplicity of $\mcR(L)$ is not explicitly known.
First, we will use the criterion of Lemma 
\ref{LemmaCriterionInitialComplex}
to verify the special case $d=1$, 
where the simplicial complex is relatively simple and the multiplicity of $\mcR(L)$ 
can be derived from results on rational normal curves \cite{Conca,Hoang}.
This is enough to deduce the general case, 
thanks to the fiber type property \cite{BrunsConcaVarbaro}
and the observation the all S-pairs reductions follow from the special cases $d=1$, $c \leq 5$, or  Theorem \ref{TheoremSpecialFiber}
(with one  exception).

\subsection{The case $d=1$}\label{SubsectionReesCaseD1}
In this subsection, we use Notation \ref{NotationBlowupSpecial} where additionally we fix $d=1$, so $c=e$.
For simplicity, we rename $y_{1,1}=y_1, y_{2,1}=y_2$, so $L $ is the ideal of minors of 
$$
\bfL=
\begin{pmatrix}
y_1 & x_1 & x_2 & \cdots & x_{e-1} & x_{e}   \\
y_2 &x_2 & x_3 & \cdots & x_{e} & 0 
\end{pmatrix}.
$$

Let $\Delta_\mcR$ be the flag simplicial complex defined by all the squarefree quadratic monomials determined in  Proposition \ref{PropLeadingMonomials}, that is, the complex whose Stanley-Reisner ideal is
\begin{align*}
\mcI(\Delta_\mcR) =& 
\big(
T_{\al,\ga} T_{\be,\de}\, \mid 1\leq \alpha  < \be < \ga < \de \leq e+1\big)+
\big(
T_{\al,\be} T_{\ga,\de}\, \mid 3\leq \alpha  < \be \leq \ga < \de \leq e+1\big)
\\
&
+\big(x_{\ga} T_{\al,\be}\, \mid \, 
1 \leq \al < \be \leq \ga \leq e,\, \be \geq 3\big) + \big(x_1 T_{1,\ga}\, \mid \, 3 \leq \ga \leq e+1\big)
\\
&
+ \big(y_{2} T_{\be,\ga}\, \mid \,  2 \leq \be < \ga \leq e+1\big) +
\big(x_{\al} T_{\be,\ga}\, \mid \,  2 \leq \al < \be < \ga \leq e+1\big). 
\end{align*}

\begin{prop}\label{PropositionDeltaReesD1}
If $e \geq 4$ then
 $\Delta_\mcR$
has $2^{e+2}-(e+1)^2-3$ facets, each containing $e+3$ elements.
\end{prop}
\begin{proof}
Let 
$
\mcW = \{y_1, y_2, x_1,\ldots, x_e, T_{1,2}, T_{1,3}, \ldots, T_{e,e+1}\}
$
be the vertex set of $\Delta_\mcR$.
Since $y_1$ and $T_{1,2}$ do not appear in any generator of $\mcI(\Delta_\mcR)$, 
they  belong to  every facet. 
If a face $F$  contains $x_h$ for some $h>1$, 
then $F$ does not contain   any $T_{\al,\be}$ with  $1 \leq \al < \be \leq h$ and $(\al,\be) \ne (1,2)$, nor any  $T_{\be,\ga}$ with $h <\be  <\ga \leq e+1$.
It follows that if $F$ contains $x_{h_1}, x_{h_2}$ with $1 < h_1<h_2$, then $F\cup \{x_{h_1+1}, \ldots, x_{h_2-1}\}$ is also a face.

We will now count the   facets according to 5 different cases.
Let $F$ be a facet.

\underline{Case 1:  $y_2 \in F$.} 
Then,  $T_{\be,\ga}\notin F$ for all $2 \leq \be < \ga \leq e+1$,
so any $T$-vertex appearing in $F$ must be of the form 
$T_{1,\varepsilon}$.
If $T_{1,\beta} \in F$ for some $\beta >2$, then 
$x_\gamma \notin F$ for all $\gamma \geq \beta$;
it follows that if $\beta \leq e$ then $F \cup \{T_{1,\beta+1}\}$ is also a face, so $T_{1,\beta+1}\in F$.
If $T_{1,\beta} \notin F$ for all $\beta > 2$, then
$F\cup\{x_1\}$ is also a face, so $x_1 \in F$.
We conclude that the facets are of the form 
$\{y_1,y_2,T_{1,2},x_2,\dots,x_{\be-1},T_{1,\be},\dots,T_{1,e+1}\}$ for $3 \leq \be \leq e+1$ and $\{y_1,y_2,T_{1,2},x_1,\dots,x_e\}$. 
This gives a total of $e$ facets, all of size $e+3$.

For the remaining cases we assume  $y_2 \notin F$.

\underline{Case 2:  $x_h \notin F$ for every $h$.} 
The only relevant monomials are those  involving the $T$-variables only. 
It follows that $F=\{y_1\}\cup F'$ where $F'$ is a facet of 
$\Delta_\mcF$ for the KW type $(1;e)$. 
By Proposition \ref{PropCaseE4}, all such facets $F$ have size $e+3$,
and by Proposition \ref{PropCountFacets}
their number  is 
$$
\sum_{h=1}^{e-2}{e+1 \choose h+1} -(e-2)(e+1)=
2^{e+1} -2 - 2(e+1) -(e-2)(e+1)
=
2^{e+1}-2-e(e+1).
$$

\underline{Case 3:  $x_1 \in F$ and $x_h \notin F$ for every $h>1$.} 
Then, $T_{1,\ga}\notin F$ for all $3 \leq \ga \leq c$. 
For any $T_{\al,\be}\in F$ with $(\al,\be) \ne (1,2)$ we 
have $(\al,\be) \subseteq (2,e+1)$;
it follows that  $F=\{x_1,y_1,T_{1,2}\} \cup F'$ where $F'$
corresponds to a facet of $\Delta_\mcF$ for the KW type $(\emptyset;e)$. 
Again, 
by Proposition \ref{PropCaseE4} all such facets $F$ have size $e+3$,
and by Proposition \ref{PropCountFacets}
their number  is 
$
\sum_{h=1}^{e-2}{e -1\choose h} -(e-2)=
2^{e-1} -2 - (e-2)
= 2^{e-1}-e.
$

\underline{Case 4:  $x_1 \notin F$ and  $x_h\in F$ for some $h>1$.} 
As noted in the beginning, 
the $x$-vertices contained in $F$ are  $x_{a+1},\dots,x_{a+b}$,
for some $1 \leq b \leq e-1, 1\leq a \leq e-b$.
Then, for each  $T_{\al,\be}\in F\setminus\{T_{1,2}\}$ we have $\alpha \leq a+1$ and $ \beta > a+ b$,
and by condition $(\dagger)$ in Section \ref{SectionSpecialFiber}
the intervals $(\al,\be)$ form a (saturated) chain with respect to inclusion.
Counting the number of such chains is simple: 
we start at $T_{1,e+1}$, 
 then we  choose either $T_{2,e+1}$ or $T_{1,e}$, and so on,
until  we choose $T_{\al,\be}$ such that $\be - \al = b$.
For each $b$ we make   $e-b$ choices, 
and we  exclude the set of choices $T_{1,e+1}\to T_{1,e}\to  \cdots \to  T_{1,b+2} \to T_{1,b+1}$
since we must end at $T_{a+1,a+b+1}$ for some $a \geq 1$.
To summarize,  for each $1 \leq b \leq e-1$ we have $2^{e-b} -1$ facets, for a total of 
$
\sum_{b=1}^{e-1} (2^{e-b}-1) = (2^e - 2) - (e-1) = 2^e-e-1
$
facets.
Each facet contains the vertices $T_{1,2},y_1$ and then  $b$ of the $x$-vertices and $e-b+1$ more $T$-vertices,
for a total of $e+3$ vertices.

\underline{Case 5:  $x_1 \in F$ and $x_h \in F$ for some $h>1$.} 
The $x$-vertices  in $F$ are now $x_1, x_{a+1}, x_{a+2}, \ldots, x_{a+b}$
for some $1 \leq b \leq e-1, 1\leq a \leq e-b$.
We argue as in Case 4 to count these facets, 
the main difference is that $T_{1,e+1}\notin F$ and instead we start from $T_{2,e+1}$. 
This time we have there are $2^{e-1- b}$ sets of choices  for each $1 \leq b \leq e-1$,
and all of them are valid,
for a total of 
$
\sum_{b=1}^{e-1} 2^{e-1-b} = 2^{e-1} -1
$
facets.
Trading $T_{1,e+1}$ for $x_1$, we see again that all facets have $e+3$ elements.

In conclusion, the total number of facets of $\Delta_\mcR$ is
$$
e + \big(2^{e+1} -2-e(e+1)\big) + \big(2^{e-1}-e\big) + 
\big(2^{e} - e-1\big) + \big(2^{e-1}-1\big) = 2^{e+2} - (e+1)^2 - 3. \qedhere
$$
\end{proof}

\begin{proof}[Proof of Theorem \ref{TheoremReesAlgebraGB} for  $d=1, e \geq 4$]
It follows from Propositions 
\ref{PropositionDeltaReesD1}, \ref{PropositionMultiplicityReesD1},
and Lemma \ref{LemmaCriterionInitialComplex}.
\end{proof}

\subsection{The general case}
In this subsection, we complete the proof of Theorem \ref{TheoremReesAlgebraGB}.

\begin{prop}\label{PropositionFiberType}
The Eagon-Northcott,
Plucker and Laplace relations \eqref{EqUpperEagonNorthcott},
\eqref{EqLowerEagonNorthcott},
\eqref{EqPluckerRelation},
\eqref{EqLaplaceNoLastColumn},
\eqref{EqLaplaceLastColumn}
generate the defining ideal $\ker(\pi_\mcR)$ of the Rees algebra $\mcR(L)$.
\end{prop}

\begin{proof}
It follows from Theorem \ref{TheoremSpecialFiber}, the fiber type property \cite[Theorem 3.7]{BrunsConcaVarbaro},
and the well-known fact that the Eagon-Northcott relations generate all the first syzygies of a 2-determinantal ideal.
\end{proof}

Thus, 
by Buchberger's criterion, it suffices to show that  all S-pairs reduce to 0.
It is known that an S-pair $S(F,G)$ reduces to 0 if $\gcd\big(\LM(F),\LM(G)\big)=1$.
Since $\ker(\pi_\mcR)$ is generated by quadrics, 
we then only need to analyze \emph{linear} S-pairs,
i.e. S-pairs of the form $S(F,G)= aF-bG$ where $a,b$ are terms of degree 1.
We are going to show  that various subsets of the generators of $\ker(\pi_\mcR)$ are themselves Gr\"obner bases;
we have already seen one instance in Theorem \ref{TheoremSpecialFiber}.

\begin{lemma}\label{LemmaGrobnerPluUen}
The upper Eagon-Northcott  
\eqref{EqUpperEagonNorthcott}
 and 
Plucker relations 
\eqref{EqPluckerRelation}
form a Gr\"obner basis.
\end{lemma}

\begin{proof}
Let $\mathcal{G}=\{\PLU_{\al,\be,\ga,\de}\mid \, 1 \leq \al < \be < \ga < \de \leq c+1\}
\cup
\{\UEN_{\al,\be,\ga}\mid \, 1 \leq \al < \be < \ga \leq c+1 \}$.
The polynomials in $\mathcal{G}$ are homogeneous with respect to the $\mathbb{N}^{c+1}$-multigrading  $\mdeg(T_{\al,\be})=\bfe_\al+\bfe_\be$,
$\mdeg(y_{1,\al})=\bfe_\al$, and 
$\mdeg(x_\al)=\bfe_{\al+d}$,
which keeps track of the column indices of the relations.

With the help of \cite{M2},  we verify the statement  when $c=5$,
that is, for the matrices of KW types
$(1^6;\emptyset), (1^5;1), (1^4;2), (1^3;3), (1^2;4),  (1;5), (\emptyset;6)$.
We perform the computation over the base ring $\kk=\mathbb{Z}$, which implies it for any field of any characteristic.
But these 7 KW types imply the arbitrary case, since they contain all possible column configurations for linear S-pairs.
More specifically,  consider the arbitrary case, we show that any S-pair in the set $\mathcal{G}$ reduces to 0 modulo $\mathcal{G}$.
By Theorem \ref{TheoremSpecialFiber},
we do not need to consider S-pairs among Plucker relations.
If we have a linear S-pair   $S(\PLU_{\al,\be,\ga,\de},\UEN_{\varepsilon,\phi, \psi})$, then the $T$-variable in $\LM(\UEN_{\varepsilon,\phi, \psi})$ also appears in 
$\LM(\PLU_{\al,\be,\ga,\de})$, 
so the number $p$ of distinct  column indices  among $\al,\be,\ga,\de,\varepsilon,\phi, \psi$ 
is either $p=4$ or $p=5$.
Since the variable $x_1$ behaves differently from the others in the monomial order $\preceq$, 
we need to distinguish whether or not the column $d+1$ appears among the $p$ indices.
In any case, this configuration of columns of $\bfL$ already occurs for one of the  7 matrices $\bfL'$ with 6 columns.
Since the claim holds for $c=5$, there exists an equation of reduction  of the S-pair to 0 for $\bfL'$,
and this equation can be chosen to be homogeneous with respect to $\mdeg(\cdot)$, in other words,
containing only relations involving the $p$ column indices. 
The same equation holds for the original $p$ columns of $\bfL$,
so $S(\PLU_{\al,\be,\ga,\de},\UEN_{\varepsilon,\phi, \psi})$ 
reduces to 0 modulo $\mathcal{G}$.
The case of the linear S-pair  $S(\UEN_{\al,\be,\ga},\UEN_{\varepsilon,\phi, \psi})$
 is completely analogous.
\end{proof}

\begin{example}
We illustrate the idea of Lemma \ref{LemmaGrobnerPluUen} 
with an example in the case $(d,e) = (3, 7)$.
Let
$$
\bfL_{3,7}
=
\begin{pmatrix}
y_{1,1} & y_{1,2} &   y_{1,3} & x_1 & x_2 & x_3 & x_4 & x_5 & x_6 & x_7   \\
y_{2,1} & y_{2,2} & y_{2,3} &x_2 & x_3 & x_4 & x_5 & x_6 & x_7 &  0 
\end{pmatrix},
$$
and consider the S-pair  between
$\PLU_{2,3,5,8}$ and $\UEN_{2,5,10}$.
The leading terms are   $\LT(\PLU_{2,3,5,8}) = -T_{2,5}T_{3,8}$
and $\LT(\UEN_{2,5,10}) = x_7 T_{2,5}$,
cf. Section \ref{SectionBlowupAlgebras}.
Thus, the S-pair is
\begin{align*}
f=S(\PLU_{2,3,5,8},\UEN_{2,5,10})  &=  T_{3,8} \UEN_{2,5,10} + x_7 \PLU_{2,3,5,8}\\
&= y_{1,2}T_{5,10}T_{3,8}-x_2T_{2,10}T_{3,8}+x_7T_{2,3}T_{5,8}+x_7T_{2,8}T_{3,5}.
\end{align*}
The  cubic polynomial $f$ is homogeneous with respect to the multigrading $\mdeg(\cdot)$, specifically,
$$
\mdeg(f)
= \bfe_2+\bfe_3+\bfe_5+\bfe_8+\bfe_{10} \in \mathbb{N}^{10},
$$
indicating that each monomial in the support of $f$ covers the columns 2, 3, 5, 8, 10 of $\bfL_{3,7}$.
Moreover, the variables $y_{i,j},x_h$ in $f$ come from  the first row of $\bfL_{3,7}$.
It follows that $f$ can also be realized as the S-pair 
$f'=S(\PLU_{1,2,4,5},\UEN_{1,4,6})$
for the matrix
$$
\bfL'
=
\begin{pmatrix}
y_{1,2} &   y_{1,3} &x_1 &  x_2 & x_5 &  x_7   \\
y_{2,2} & y_{2,3} &x_2& x_5 &x_7  &  0 
\end{pmatrix},
$$
since the entries in the second row do not play any role. 
More precisely, $f$ is obtained from $f'$ by changing the indices of the $T_{\al,\be}$ by 
 $1 \to 2,\, 2\to 3,\, 4\to 5,\, 5\to 8,\, 6 \to 10$.
Up to renaming variables, $\bfL'$ can be identified with the matrix $\bfL_{2,4}$, obtained for $(d,e)=(2,4)$.
Moreover, it follows from  Subsection \ref{SubsectionMonomialOrder} that, 
under this identification,
the monomial order used for $\bfL'$, in the case $(d,e)=(2,4)$, is a restriction of the 
monomial order used for $\bfL_{4,7}$, in the case $(d,e)=(3,7)$.
Now, assuming that Pl\"ucker relations and upper Eagon-Northcott relations of $\bfL'$ form a Gr\"obner basis,
there is an $\mdeg$-homogeneous equation of reduction of $f'$ to zero. 
The fact that it is $\mdeg$-homogeneous indicates that this equation only involves the  columns 1, 2, 4, 5, 6 of $\bfL'$. 
Specifically, in this example, one such equation is 
$$
f' = S(\PLU_{1,2,4,5},\UEN_{1,4,6})  = 
T_{1,6}\UEN_{2,4,5}+ T_{2,4}\UEN_{1,5,6}+ 
T_{4,5}\UEN_{1,2,6}-y_{1,2}\PLU_{2,4,5,6}.
$$
But we can consider the corresponding equation for the matrix $\bfL_{3,7}$,
 by changing the column indices
 $1 \to 2,\, 2\to 3,\, 4\to 5,\, 5\to 8,\, 6 \to 10$,
and we obtain  an equation of reduction to 0 for the original S-pair:
$$
f = S(\PLU_{2,3,5,8},\UEN_{2,5,10})= T_{2,10}\UEN_{3,5,8}+ T_{3,5}\UEN_{2,8,10}+ T_{5,8}\UEN_{2,3,10}-y_{1,2}\PLU_{3,5,8,10}.
$$
Note that the reason we included the third column   in $\bfL'$ is to guarantee that the monomial order for $\bfL'$ is a restriction of the one for $\bfL_{3,7}$,
since  $x_1$ behaves differently from $x_2, \ldots, x_e$, cf. Subsection \ref{SubsectionMonomialOrder}.
\end{example}

\begin{lemma}\label{LemmaGrobnerPluLen}
The lower Eagon-Northcott  
\eqref{EqLowerEagonNorthcott}
 and 
Plucker relations 
\eqref{EqPluckerRelation}
form a Gr\"obner basis.
\end{lemma}

\begin{proof}
The proof proceeds exactly as the one of Lemma \ref{LemmaGrobnerPluUen},
the only modifications being the column multigrading, 
where we set $\mdeg(x_\al)=\bfe_{\al+d-1}$ instead of $\mdeg(x_\al)=\bfe_{\al+d}$,
and  the fact that the distinguished column is not $d+1$ but $c+1$,
since its entry on the second row is $0$.
\end{proof}

The simple argument of Lemmas \ref{LemmaGrobnerPluUen}, \ref{LemmaGrobnerPluLen} 
cannot be applied to S-pairs involving Laplace relations, or
to S-pairs involving both upper and lower Eagon-Northcott relations,
 due to the absence of a  suitable column multigrading.
In fact, the reduction of these S-pairs will typically extend beyond the columns originally involved in the S-pair.
However, we can reduce all but one of these to the case $d=1$.

Observe that the linear S-pairs among the polynomials of Theorem \ref{TheoremReesAlgebraGB} have bidegree $(0,3), (1,2)$ or $(2,1)$.
The ones of bidegrees $(0,3)$ are exactly those involved in Theorem \ref{TheoremSpecialFiber}.

\begin{lemma}\label{LemmaSpair12}
All  S-pairs of bidegree $(1,2)$ reduce to 0 modulo $\mathcal{G}$.
\end{lemma}
\begin{proof}
There are three kinds of  S-pairs of bidegree $(1,2)$:
those of the form 
 $S(\LAP_{\al,\be,\ga,\de},\UEN_{\varepsilon,\phi,\psi})$
 and
  $S(\LAP_{\al,\be,\ga,\de},\LEN_{\varepsilon,\phi,\psi})$, 
 where the two leading monomials share a $T$-variable,
and the case of $S(\LEN_{\al,\be,\ga},\UEN_{\varepsilon,\phi,\psi})$ 
where $\gcd\big(\LM(\LEN_{\al,\be,\ga}),\LM(\UEN_{\varepsilon,\phi,\psi})\big)=x_h$ for some $h$.
Inspecting  Proposition \ref{PropLeadingMonomials},
we see that, in each case, at most one scroll column is involved in the S-pair, while all the other columns belong to the Jordan block.
It follows that this configuration of columns of $\bfL$ already occurs in a matrix $\bfL'$    with KW type $(1;1^e)$ for some $e \gg 0$.
By Subsection \ref{SubsectionReesCaseD1},
there exists an equation of reduction of the S-pair to 0 for $\bfL'$,
and the same equation holds for the original S-pair in $\bfL$.
\end{proof}

\begin{lemma}\label{LemmaSpair21}
All  S-pairs of bidegree $(2,1)$ reduce to 0 modulo $\mathcal{G}$.
\end{lemma}
\begin{proof}
The  only kind of S-pair of bidegree $(2,1)$ is the case of 
$S(\UEN_{\al,\be,\ga},\LEN_{\varepsilon,\phi,\psi})$ 
where the  leading monomials have a common $T$-variable.
Recall that $\LM(\LEN_{\varepsilon,\phi,\psi})= \ell_{2,\varepsilon}T_{\phi,\psi}$,
while
$\LM(\UEN_{\al,\be,\ga})$ is  $\ell_{1,\ga} T_{\al,\be}$  if $\be \geq d+2$ and
 $\ell_{1,\be}T_{\al,\ga}$ if $\be \leq d+1$,
cf. Proposition \ref{PropLeadingMonomials}.
It follows that there are exactly four distinct indices, say $\sigma_1 < \sigma_2 < \sigma_3 < \sigma_4$.
We must have $\varepsilon = \sigma_1 , \phi = \al = \sigma_2$, and then 
 $  \be =\psi= \sigma_3, \ga = \sigma_4$
 if $\be \geq d+2$,
while 
$\be = \sigma_3, \gamma = \psi = \sigma_4$
if $\be \leq d+1$.
Moreover, we may assume that $\sigma_2 \leq d$, i.e.
that at least two of the four columns  are scroll columns,
arguing as in Lemma \ref{LemmaSpair12}.
Thus, $\ell_{2,\varepsilon}=y_{2,\sigma_1}$.

Consider the following second syzygy in the Eagon-Northcott resolution of the ideal $L$:
\begin{equation}\label{EqSecondSyzygyEN}
\sum_{i=1}^4
(-1)^{i-1}
\ell_{1,\sigma_i}\LEN_{\{\sigma_1,\sigma_2,\sigma_3,\sigma_4\}\setminus\{\sigma_i\}}
= 
\sum_{i=1}^4
(-1)^{i-1}
\ell_{2,\sigma_i}\UEN_{\{\sigma_1,\sigma_2,\sigma_3,\sigma_4\}\setminus\{\sigma_i\}}.
\end{equation}
The equation arises from the $4 \times 4$ determinant obtained by duplicating both the first and last  row of the $2\times 4$ submatrix $(C_{\sigma_1},C_{\sigma_2},C_{\sigma_3},C_{\sigma_4})$ of $\bfL$.
We  analyze the highest monomials in \eqref{EqSecondSyzygyEN}
with respect to $\preceq$.
By   step (1) in  Subsection \ref{SubsectionMonomialOrder},
the highest three monomials in either side are those containing 
the best variable $y_{2,\sigma_1}$.
In the right-hand side 
of \eqref{EqSecondSyzygyEN}
these are simply the terms of 
$\ell_{2,\sigma_1}\UEN_{\sigma_2,\sigma_3,\sigma_4} = 
\ell_{2,\sigma_1}\UEN_{\al,\be,\ga}$.
In the left hand-side 
they are  the three monomials  
$$
\ell_{1,\sigma_i} 
\LM(\LEN_{\{\sigma_1,\sigma_2,\sigma_3,\sigma_4\}\setminus\{\sigma_i\}})
= 
\ell_{2,\sigma_1}\ell_{1,\sigma_i} T_
{\{\sigma_2,\sigma_3,\sigma_4\}\setminus\{\sigma_i\}}
\quad
\text{for }i = 2,3,4.
$$
Step (2) breaks no tie among  these monomials. Then, step (3) selects 
$\ell_{2,\sigma_1}\ell_{1,\sigma_3} T_
{\sigma_2,\sigma_4}$
and
$ \ell_{2,\sigma_1}\ell_{1,\sigma_4}T_
{\sigma_2,\sigma_3}$,
and finally, step (4) selects 
$\ell_{2,\sigma_1}\ell_{1,\sigma_3} T_
{\sigma_2,\sigma_4}$
if $\sigma_3 \leq d+1$ and
 selects
$\ell_{2,\sigma_1}\ell_{1,\sigma_4} T_
{\sigma_2,\sigma_3}$
if $\sigma_3 \geq d+2$.
Comparing with the first paragraph,  the unique highest monomial in the left hand-side  is the one containing  $\LM(\LEN_{\varepsilon, \phi, \psi})$,
while 
the second highest monomial comes from a different 
$\LEN_{\{\sigma_1,\sigma_2,\sigma_3,\sigma_4\}\setminus\{\sigma_i\}}$.

To summarize, the highest monomial in either side of \eqref{EqSecondSyzygyEN}
occurs exactly once, respectively in $\LEN_{\varepsilon,\phi,\psi}$  and in 
$\UEN_{\al,\be,\ga}$.
The second highest monomial also occurs  once;
in the right-hand side it occurs in 
$\UEN_{\al,\be,\ga}$,
while in the left-hand side it does not occur in $\LEN_{\varepsilon,\phi,\psi}$.
Rearranging the addends of \eqref{EqSecondSyzygyEN}
we obtain an equation of reduction for $S(\UEN_{\al,\be,\ga},\LEN_{\varepsilon,\phi,\psi})$:
the highest monomial cancels in the S-pair, while the second highest monomial does not cancel in the S-pair and  becomes the leading monomial, dominating all the remaining terms in the equation.
\end{proof}

\begin{proof}[Proof of Theorem \ref{TheoremReesAlgebraGB}]
By Theorem \ref{TheoremSpecialFiber} and Lemmas \ref{LemmaSpair12}, \ref{LemmaSpair21},
we see that all S-pairs reduce to 0.
\end{proof}

\section{Cohen-Macaulayness and proof of the main theorem}

The degenerations of blowup algebras established in 
Section \ref{SectionDegenerations}
allow one to deduce their Cohen-Macaulayness essentially for free.
In fact, combining with the results of Sections \ref{SectionSpecialFiber}, and \ref{SectionReesAlgebra},
we deduce that the blowup algebras of any $I \in \mcH_{c,d}$ admit a sequence of degenerations  to  Stanley-Reisner rings.
Since Stanley-Reisner rings are cohomologically full,
 it follows from  Lemma \ref{LemmaCohomologicallyFullDegenerations} 
 that if \emph{some} algebra in this sequence  is Cohen-Macaulay, then \emph{every}  algebra is Cohen-Macaulay.
For $d >0 $, we may use the results on balanced rational normal scrolls from \cite{CHV},
while the case $d=0$ follows easily from $d=1$, thanks to the following combinatorial observation.

Denote by  $\Delta_\mcR(d,e)$, resp. $\Delta_\mcF(d,e)$,
 the initial complex of the Rees algebra $\mcR(L)$, resp. special fiber ring $\mcF(L)$,
associated to the KW type $(1^d;e)$.
Given a face $F$ of a simplicial complex $\Delta$, 
the \emph{link} of $F$ in $\Delta$ is the subcomplex 
$\mathrm{link}_\Delta(F) = \{ G \in \Delta \, | \, F \cup G \in \Delta, F\cap G = \emptyset\}$.

\begin{lemma}\label{LemmaLinks}
The complex $\Delta_\mcR(0,e)$ is isomorphic to a link of $\Delta_\mcR(1,e)$,
and $\Delta_\mcF(0,e)$ is isomorphic to a link of $\Delta_\mcF(1,e)$.
\end{lemma}
\begin{proof}
We identify $\Delta_\mcR(0,e)$ with the subcomplex of $\Delta_\mcR(1,e)$ induced by the last $e$ columns of $\bfL(1,e)$,
i.e. by the vertex subset $\{x_1, \ldots x_e\}\cup \{T_{\al,\be}\, \mid \, 2 \leq \al < \be \leq e+1\}$. 
In other words, all  column indices are shifted by 1 in $\Delta_\mcR(0,e)$.

First, we observe that $x_1$ and $T_{2,e+1}$ appear in every facet of $\Delta_\mcR(0,e)$,
since they do not appear in the  generators of the Stanley-Reisner ideal
$\mcI\big(\Delta_\mcR(0,e)\big)$
by Theorem \ref{TheoremReesAlgebraGB} and Proposition \ref{PropLeadingMonomials}.
Now, 
consider the face $F = \{ x_1, T_{2,e+1}\} \in \Delta_\mcR(1,e)$.
Inspecting the generators of $\mcI\big(\Delta_\mcR(1,e)\big)$,
or the description of the facets in Proposition
\ref{PropositionDeltaReesD1},
we see that if  $F'\in \Delta_\mcR(1,e)$ is a facet such that $F \subseteq F'$, then 
$T_{1,2} \in F'$ but $T_{1,\gamma} \notin F'$ for  $\ga \geq 3$,
and $y_1 \in F'$ but $y_2\notin F'$.
In other words,
$F' \subseteq \{y_1,T_{1,2}\}\cup G$ for some  $G \in \Delta_\mcR(0,e)$.
Conversely, for every  $G \in \Delta_\mcR(0,e)$
we have $G \cup \{y_1, T_{1,2}\}\in\Delta_\mcR(1,e)$,
since $y_1, T_{1,2}$ belong to every facet of $\Delta_\mcR(1,e)$.
We conclude that 
$
\mathrm{link}_{\Delta_\mcR(1,e)}F
= \big\{ \{y_1, T_{1,2}\} \cup G \setminus\{x_1,T_{2,e+1}\}\, \mid \, G \in \Delta_\mcR(0,e)\big\}.
$
In particular, $\Delta_\mcR(0,e) \cong \mathrm{link}_{\Delta_\mcR(1,e)}F$, since  $\Delta_\mcR(0,e)$ is obtained from $
\mathrm{link}_{\Delta_\mcR(1,e)}F$ by exchanging $y_1, T_{1,2}$ with $x_1,T_{2,e+1}$.
Restricting to only the $T$-vertices we obtain the statement for 
$\Delta_\mcF(0,e)$ and $\Delta_\mcF(1,e)$.
\end{proof}

Now, we are ready to  combine all the results of the paper and prove our main theorem.

\begin{proof}[Proof of Theorem \ref{TheoremMain}]
Let $S=\mathrm{Sym}(\kk^{n+1})$ be a polynomial ring over a field $\kk$,
and let  $I \subseteq S$ be
the ideal of minors of a $2 \times (c+1)$ matrix $\mathbf{M}$ of linear forms with $\codim(I) = c$.
If $c =1$, 
then $I$ is a principal ideal, thus, $\mcR(I)$ and $\mcF(I)$ are polynomial rings,
and the conclusions of Theorem \ref{TheoremMain} are straightforward.
Therefore, we may assume $c \geq 2$.
The formation of the blowup algebras and the Koszul and Cohen-Macaulay properties are compatible with  extensions of the base field and with the adjunction of new variables.
Therefore,  we may assume $\kk = \overline{\kk}$ and
 $ \codim\, I_1(\mathbf{M})= n+1$.

By Proposition \ref{PropKWNormalFormMaximalCodimension},
there are two possibilities,
corresponding to having $c \leq n $ or $c = n+1$.
If $c = n+1$, then it follows by Proposition \ref{PropKWNormalFormMaximalCodimension} that $I$ is the square of the maximal ideal  of $S$. 
The conclusions of Theorem \ref{TheoremMain} are well-known in this case, since $\mcF(I)$ and $\mcR(I)$ are particularly simple toric rings.
For instance, 
it follows by \cite[Theorem 2.6]{DeNegri}  and \cite[Theorem 5.1]{HHV} that they are defined by square-free quadratic Gr\"obner bases, in particular, that they are Koszul algebras,
while the Cohen-Macaulay property follows by \cite[Theorem 5.4]{BrunsConca}.

If $c \leq n$,
then  $I \in \mcH_{c,d}$ for  $d=n-c$.
Assume first that $d >0$.
Then, there is a balanced rational normal scroll $J \in \mcH_{c,d}$. 
Applying Lemmas \ref{LemmaDegenerationPeelingScrollColumn} and \ref{LemmaDegenerationTwoJordanBlocks},
we find a sequence of ideals $\{I_{(i)}\}_{i=1}^q$ in $\mcH_{c,d}$ such that $I_{(1)}=J, I_{(p)}=I, I_{(q)}=L$ for some $1 \leq p \leq q$, and $I_{(i)}$ degenerates to $I_{(i+1)}$ for all $i$.
By Corollary \ref{CorDegenerationBlowupAlgebras}, $\mcR(I_{(i)})$ degenerates to $\mcR(I_{(i+1)})$ for all $i$.
By Theorem \ref{TheoremReesAlgebraGB}, 
there is a further degeneration of $\mcR(I_{(q)})$ to the quadratic Stanley-Reisner ring $	\mathcal{S}$ of the simplicial complex $\Delta_\mcR(d,e)$.
Since $\mathcal{S}$ is Koszul by \cite[Remark 6]{CDR},
we can apply Lemma \ref{LemmaSingularitiesDegeneration} (2) repeatedly and obtain that $\mcR(I_{(i)})$ is Koszul for every $i$.
Since $\mathcal{S}$ is cohomologically full  by  \cite[Remark 2.5]{DDM},
we can apply  Lemma \ref{LemmaSingularitiesDegeneration} (3) repeatedly and obtain that $\mcR(I_{(i)})$ is cohomologically full for every $i$. 
Since $\mcR(I_{(1)})$ is Cohen-Macaulay by \cite[Theorem 3.8]{CHV}, 
we can apply  \ref{LemmaCohomologicallyFullDegenerations} repeatedly and obtain that all $\mcR(I_{(i)})$ and $\mathcal{S}$ are  Cohen-Macaulay.
The same argument works for the special fiber rings.
 
Finally, assume that  $d=0$, that is, $c=n$.
In this case, there are no rational normal scrolls in $\mcH_{c,0}$.
By the previous paragraph, the complexes $\Delta_\mcR(1,e)$ and $\Delta_\mcF(1,e)$ are Cohen-Macaulay.
By Lemma \ref{LemmaLinks} and \cite[Corollary 8.1.8]{HerzogHibi}, it follows that $\Delta_\mcR(0,e)$ and $\Delta_\mcF(0,e)$ are also Cohen-Macaulay.
Now, applying Lemma \ref{LemmaSingularitiesDegeneration} as in the previous paragraph, 
it follows that $\mcR(I)$ and $\mcF(I)$ are Koszul, Cohen-Macaulay, and cohomologically full for every $I \in \mcH_{c,d}$.
\end{proof}

We record some of the facts we proved on the way to proving Theorem \ref{TheoremMain}.

\begin{cor}
The simplicial complexes $\Delta_\mcR$ and $\Delta_\mcF$ are Cohen-Macaulay.
\end{cor}

\begin{cor}
Let $I$ be a 2-determinantal ideal. Then 
 $\mcR(I)$ and $\mcF(I)$ are cohomologically full.
\end{cor}

As another byproduct,
 Theorem \ref{TheoremMain}
 also implies that the formulas for the reduction number of $I$ found in \cite[Corollary 4.8]{CHV}, in the case of balanced scrolls, are valid for any 2-determinantal ideal.
In fact, if  $\mcF(I)$ is Cohen-Macaulay, 
then  the reduction number of $I$ coincides with the degree of the numerator of the Hilbert series of 
 $\mcF(I)$.
The reduction number plays an important role in the theory of residual intersections, see e.g.
\cite{EisenbudHunekeUlrich,EisenbudUlrich}
for  recent progress closely related to the themes of this paper.

\section{Questions}

We conclude by
discussing  some potential future directions and open problems suggested by our work.

In the special case of balanced rational normal scrolls,
a stronger statement than Theorem \ref{TheoremMain} is proved in \cite[Theorem 3.8]{CHV}:
their blowup algebras have rational singularities if $\mathrm{char}(\kk) =0$ and
$F$-rational singularities if 
$\mathrm{char}(\kk) >0$.
These properties are stronger than Cohen-Macaulayness, and they also imply normality;
it is therefore natural to ask whether they hold for all 2-determinantal ideals.
Since these singularities are preserved by deformation, 
the machinery of our paper applies, 
and it suffices to study the blowup algebras of the most special ideal $L$.

\begin{question}
Let $L$ be the 2-determinantal ideal of KW type $(1^d;e)$.
Are the blowup algebras $\mcR(L), \mcF(L)$ normal?
Do they have rational or $F$-rational singularities?
\end{question}
\noindent
We point out that the main technique of \cite{CHV},
namely degenerations to toric rings by means of Sagbi bases,
cannot be applied.
In fact, in general the natural generators of $L$ do not 
form a Sagbi basis for any monomial order, 
and the same is true for non-balanced scrolls.

The most interesting next step in this area is perhaps 
the investigation of blowup algebras arising from matrices with more than two rows.
For instance, very little is known
 about their singularities or  degrees of defining relations;
see  \cite{CooperPrice} for some results in this direction.
However, in principle, 
one may 
apply the same strategy developed in this paper:
for  matrices of arbitrary size,
the corresponding determinantal ideals are still parametrized by well-behaved Hilbert schemes \cite{KleppeMiroRoig},
and the  results of \cite{BrunsConcaVarbaro} are still valid.
The main goal would then be to identify the most special  ideals in this parameter space,
analogously to Theorem \ref{TheoremUniqueMinimumHcd}.
Of particular interest would be the case of matrices of Hilbert-Burch size, that are sufficiently general in the sense 
 of \cite[Theorem 3.7]{BrunsConcaVarbaro}:
 their determinantal ideals define arithmetically Cohen-Macaulay subschemes of codimension 2,
 and their Hilbert schemes are among the  best behaved and most studied in the literature.

\begin{problem}
Fix $m, n \in \mathbb{N}$.
Let $S = \mathrm{Sym}(\kk^{n+1})$
and let $p(z)$ be the Hilbert polynomial of $S/I_m(\bfM)$ 
where $\bfM$ is an 
 $m \times (m+1)$ matrix of linear forms with $\codim\, I_m(\bfM) =2$.
Let $\mcH\subseteq \Hilb^{p(z)}(\P^n)$ be the set
 parametrizing ideals $I_m(\bfM)$
such that $\codim\, I_1(\bfM) =n+1$ and $\codim\, I_j(\bfM) \geq \min(n+1, m+1-j)$ for $2 \leq j \leq m-1$.
Find a (small) set of ideals $L \in \mathcal{H}$ such that every $I \in \mathcal{H}$ degenerates to one of them.
\end{problem}

Another interesting question is what happens to the blowup algebras when we drop the assumption that $I_m(\bfM)$ has 
the expected codimension.
This question is probably hopeless for  matrices of arbitrary size, 
but it might be reasonable in the case of matrices with 2 rows, 
where Kronecker-Weierstrass normal forms are still available.
An interesting case is that  of Kronecker-Weierstrass normal forms without nilpotent blocks, but with possibly repeated eigenvalues in the Jordan blocks.
Computational evidence suggests that the defining relations may still be  similar to those found in Section \ref{SectionBlowupAlgebras}.

\begin{question}
Let $I = I_2(\bfM)$ where 
 $\bfM$ is a concatenation of scroll blocks and Jordan blocks.
Are the blowup algebras $\mcR(I)$ and $\mcF(I)$ defined by quadratic relations? 
\end{question}

We conclude with a combinatorial consideration.
For each $c,d$ there are two 
different known squarefree degenerations of the special fiber ring $\mcF(I)$ of some  
 $I \in \mcH_{c,d}$:
 one was established in \cite{CHV}, while the other follows from Theorem \ref{TheoremSpecialFiber}.
It would be interesting to compare the combinatorics and homology of the two simplicial complexes $\Delta_{CHV}$ and $\Delta_\mcF$ determined by these two degenerations.
For example, computational evidence suggests the following question about the graded Betti numbers of their Stanley-Reisner ideals:

\begin{question}
Is $\beta_{i,j}\big(\mcI(\Delta_\mcF)\big) \leq \beta_{i,j}\big(\mcI(\Delta_{CHV})\big)$ for all $i,j$?
\end{question}
\noindent
The question is already interesting in the case $c -d \leq 4$, in which  there are only Plucker  relations,
 and $\mcF(I)$ is simply the coordinate ring of the Grassmannian variety of lines: 
then, $\Delta_{CHV}$ is the non-nesting complex, and  $\Delta_\mcF $ is the non-crossing complex. 
They have been studied e.g. in \cite{PPS,SSW}.
In this case,
 $\Delta_\mcF$ is Gorenstein, while $\Delta_{CHV}$ is not,
cf. \cite{SSW}.

\section*{Acknowledgments}
The authors would like to thank Aldo Conca,
Izzet Coskun,
Alessio D'Alì, 
David Eisenbud, 
Allen Knutson,
Hop D. Nguyen, and Francesco Russo
for several helpful conversations.
Alessio Sammartano was partially supported by 
 PRIN 2020355B8Y “Squarefree Gr\"obner degenerations, special varieties and related topics”.
Computations with Macaulay2 
\cite{M2} provided valuable insights during the preparation of this paper.

\appendix

\section{Multiplicities}\label{AppendixMultiplicities}

In this appendix, we collect formulas for the multiplicities 
of various blowup algebras. 
Recall that, for each $c\geq 2$ and $ d \geq 0$,
 $\mcH_{c,d}$ is the set of 2-determinantal ideals,
in a polynomial ring in $c+d+1$ variables,
whose Kronecker-Weierstrass normal form has $c+1$ columns and  $d$ scroll blocks, cf. Definition \ref{DefHcd}.

\begin{prop}\label{PropositionInvariantsFiberD0}
Let $L \in \mcH_{c,0}$ be the ideal of KW type $(\emptyset; c+1)$.
Then $\mcF(L)$ has Krull dimension $c$ and multiplicity $2^{c}-c-1$.
\end{prop}
\begin{proof}
We compute the Hilbert function   $HF(\mcF(L);h)$ by  exploiting  
results on  powers of  ideals   of rational normal curves \cite{Conca}.
For the reader's convenience, 
we switch to the notation of \cite{Conca},
and   let $n=c+1$.
Denote by $\beta_i$ the $i$-th total Betti number,
then $HF(\mcF(L);h)= \beta_0^S(L^h)$,
 the number of  generators of $L^h$.

Let $R = \kk[x_0, \ldots, x_{n-1}], S = R[ x_{n}],$ and
let $L\subseteq R$ and $ P\subseteq S$  be the ideals generated by the minors of 
$$
\bfL = 
\begin{pmatrix}
x_0 & x_1 & \cdots & x_{n-2} & x_{n-1} \\
x_1 & x_2 & \cdots & x_{n-1} & 0 \\
\end{pmatrix}
\quad 
\text{and}
\quad
\mathbf{P} = 
\begin{pmatrix}
x_0 & x_1 & \cdots & x_{n-2} & x_{n-1} \\
x_1 & x_2 & \cdots & x_{n-1} & x_n \\
\end{pmatrix},
$$
respectively.
So $L$ is the ideal of the statement, while $P$ is the ideal of a rational normal curve.
It is clear that $L = \frac{P+(x_{n})}{(x_{n})} \subseteq \frac{S}{(x_{n})} = R$.
The  primary decomposition of the powers of $P$ is 
$P^h=P^{(h)}\cap M^{2h}$,
where $P^{(h)}$ is the  symbolic power of $P$ and $M$ is the maximal ideal of $S$  \cite[Proof of Theorem 1]{Conca}.
Since $x_{n}\notin P$,
 the module
$\frac{P^h:x_{n}}{P^h}$ has finite length,
and  by  \cite[Proof of Lemma 2.4]{BrunsConcaVarbaro}
we have the formula\footnote{We have $\be_n$ instead of $\be_{n-1}$ since  $S$ has $n+1$ variables here, while it has $n$ variables  in \cite[Lemma 2.4]{BrunsConcaVarbaro}.}
\begin{equation}\label{EqBCVBeta}
HF(\mcF(L);h)=\be^R_0(L^h) = \be^{S}_0(P^h) - \be^{S}_{n}(P^h).
\end{equation}
The number   $\be_0^S(P^h)$ was computed  in \cite[Equation (7)]{Conca}: if  $n \geq 4$ then
\begin{equation*}\label{EqBeta0Ph}
\beta_0^S(P^h) = 
{n + 2h \choose n} -(2n+1){n+h-1\choose n} -(n^2-3n+1) {n+h-2 \choose n}.
\end{equation*}
To derive the  number $\be^{S}_n(P^h)$,
we consider \cite[Equation (8)]{Conca}
\begin{equation}\label{EqPolynomialBetaPh}
z^{2h} \sum_{i=0}^n (-1)^i \beta_i^S(P^h) z^i
= 1-(1-z)^{n-1}H_h(z)
\end{equation}
where 
$$
H_h(z) =
 \sum_{i=0}^{2h-1} {n+i \choose n} z^i(1-z)^2 + z^{2h}(2he_0(h) -e_1(h))(1-z)+z^{2h}e_0(h)
$$
and $e_0(h),e_1(h)$ are two numerical functions defined in \cite[Theorem 5]{Conca}. Explicitly, they are
 $$
e_0(h) = n {n+h-2 \choose n-1} \quad \text{and}
\quad
e_1(h) = (n-1)(n+2) {n+h-2 \choose n} + (n-1) {n+h-2 \choose n-1}.
$$
Comparing the  highest term $z^{2h+n}$ in the two sides of \eqref{EqPolynomialBetaPh} we obtain
$$
(-1)^n\be_n^S(P^h) z^{2h+n} = 
-(-z)^{n-1}\left(  {n+2h-1 \choose n} z^{2h-1}(-z)^2 + z^{2h}(2he_0(h) -e_1(h))(-z)\right).
$$
Substituting $z=1$ and dividing by $(-1)^n$ we find
\begin{align*}
\be_n^S(P^h) =& 
  {n+2h-1 \choose n}  - 2he_0(h) +e_1(h)\\
  =&
    {n+2h-1 \choose n}  - 2hn {n+h-2 \choose n-1} +(n-1)(n+2) {n+h-2 \choose n} + (n-1) {n+h-2 \choose n-1}
 \\
=& 
  {n+2h-1 \choose n}  - 2(h-1)n {n+h-2 \choose n-1} -2n {n+h-2 \choose n-1}
  \\&
  +
  (n^2+n-2) {n+h-2 \choose n} + (n-1) {n+h-2 \choose n-1}
\\
=& 
  {n+2h-1 \choose n}  - 2n^2 {n+h-2 \choose n} +(n^2+n-2) {n+h-2 \choose n} - (n+1) {n+h-2 \choose n-1}
\\
=&
  {n+2h-1 \choose n} + (-n^2+n-2)  {n+h-2 \choose n}  - (n+1) {n+h-2 \choose n-1}.
\end{align*}
Substituting the  formulas of $\beta_0^S(P^h )$ and $\beta_n^S(P^h)$ in \eqref{EqBCVBeta} we obtain the Hilbert polynomial $HP(\mcF(L);h)$
\begin{align*}
& {n + 2h \choose n} -   {n+2h-1 \choose n}
+ (2n+1)\left({n+h-2 \choose n} 
-{n+h-1\choose n}\right)+(n+1) {n+h-2 \choose n-1}
\\
= & 
{n+2h-1 \choose n-1}- (2n+1){n+h-2 \choose n-1} 
+(n+1) {n+h-2 \choose n-1} =
{n+2h-1 \choose n-1}- n{n+h-2 \choose n-1}.
\end{align*}
The last formula tells us that 
 the leading term of 
$HP(\mcF(L);h)$ is
$
\frac{(2h)^{n-1} - n(h)^{n-1}}{(n-1)!},
$
implying that $\mcF(L)$ has dimension $n-1$ and multiplicity $2^{n-1}-n$ as desired.
The case $n=3$ is easily proved  by checking that the three generators of  $\mcF(L)$ are algebraically independent,
so $\mcF(L)$ is a polynomial ring.
\end{proof}

\begin{prop}\label{PropositionInvariantsFiberD1}
Let $I \in \mcH_{c,d}$ with  $d \geq 1$.
Then 
$\mcF(I)$ has Krull dimension  $\min(2c-1,c+d+1)$, and the multiplicity is 
${2c-2\choose c-1}-{2c-2\choose c}$
if $c-d\leq 2$,
it is 
$\sum_{j=2}^{c-d} {c+d \choose c+1-j} -
(c-d-1){c+d\choose c}$
if $c-d \geq 3$.
\end{prop}
\begin{proof}
The formulas follow from 
 \cite[Corollaries 4.2 and 4.5]{CHV}
 and  Lemma \ref{LemmaConstantHilbertFunctionBlowups},
since $\mcH_{c,d}$ contains exactly one balanced rational normal scroll with $c+1$ columns and $d$ scroll blocks
(we warn the reader that in \cite{CHV} 
 the number of columns is denoted by $c$ and not by $c+1$).
\end{proof}

\begin{prop}\label{PropositionMultiplicityReesD1}
Let $I \in \mcH_{c,1}$.
The multiplicity of $\mcR(I)$ is $2^{c+2}-(c+1)^2-3$.
\end{prop}
\begin{proof}
As in Proposition \ref{PropositionInvariantsFiberD1},
the formula follows from \cite[Corollary 1.7]{Hoang} and   Lemma \ref{LemmaConstantHilbertFunctionBlowups},
since $\mcH_{c,1}$ contains exactly one ideal of KW type $(c+1;\emptyset)$, which defines a rational normal curve.
\end{proof}

\end{document}